\documentclass[12pt]{amsart}

\usepackage{macros}
\standardsettings

\usepackage{charter}

\draftfalse

\newcommand{\gammacoeff}{\beta}

\newcommand{\byvar}[1]{\notevar{by #1}}
\newcommand{\sincevar}[1]{\notevar{since #1}}

\newcommand{\notevar}[1]{\;\;\;\;\text{(#1)}}

\newcommand{\pdim}{{\rm d}}
\renewcommand{\dist}{{\rm dist}}

\renewcommand{\mathbb}{\amsbb}

\begin{document}
\title[Self-affine sponges: a dimension gap]{The Hausdorff and dynamical dimensions of self-affine sponges:\\ a dimension gap result}

\authortushar\authordavid


\begin{Abstract}
We construct a self-affine sponge in $\mathbb R^3$ whose dynamical dimension, i.e. the supremum of the Hausdorff dimensions of its invariant measures, is strictly less than its Hausdorff dimension. This resolves a long-standing open problem in the dimension theory of dynamical systems, namely whether every expanding repeller has an ergodic invariant measure of full Hausdorff dimension. More generally we compute the Hausdorff and dynamical dimensions of a large class of self-affine sponges, a problem that previous techniques could only solve in two dimensions. The Hausdorff and dynamical dimensions depend continuously on the iterated function system defining the sponge, implying that sponges with a dimension gap represent a nonempty open subset of the parameter space.
\end{Abstract}
\maketitle

\tableofcontents

\section{Introduction}

A fundamental question in dynamics is to find ``natural'' invariant measures on the phase space of a dynamical system. Such measures afford a window into the dynamical complexity of chaotic systems by allowing one to study the statistical properties of the system via observations of ``typical'' orbits. For example, the knowledge that Gauss measure on $[0,1]$ is ergodic and invariant with respect to the Gauss map allows one to compute the distribution of continued fraction partial quotients of Lebesgue almost every real number \cite[\63.2]{EinsiedlerWard}. In general, ergodic invariant measures that are absolutely continuous to Lebesgue measure are often considered the most physically relevant, since they describe the statistical properties of the forward orbits of a set of points of positive Lebesgue measure.

However, in many cases there are no invariant measures absolutely continuous to Lebesgue measure. In this circumstance, there are other ways of deciding which invariant measure is the most ``natural'' -- for example, Sinai, Ruelle, and Bowen considered a class of invariant measures (now known as SRB measures) that still describe the behavior of forward orbits of points typical with respect to Lebesgue measure, even though these invariant measures are not necessarily absolutely continuous to Lebesgue measure, see e.g. \cite{Young3}. However, there are some disadvantages to this class of measures, for example we may want to consider measures supported on a fractal subset of interest such as a basic set or a repeller, and SRB measures may not be supported on such a fractal.

A complementary approach is to judge how natural a measure is in terms of its Hausdorff dimension. For example, Lebesgue measure has the largest possible Hausdorff dimension of any measure, equal to the Hausdorff dimension of the entire space. If we are looking for measures supported on a fractal subset, it makes sense to look for one whose Hausdorff dimension is equal to the Hausdorff dimension of that set. An ergodic invariant measure with this property can be thought of as capturing the ``typical'' dynamics of points on the fractal. In cases where such a measure is known to exist, it is often unique; see e.g. \cite[Theorem 9.3.1]{PrzytyckiUrbanski} and \cite[Theorem 4.4.7]{MauldinUrbanski2}, where this is proven in the cases of conformal expanding repellers and conformal graph directed Markov systems, respectively.

On the other hand, if the Hausdorff dimension of an invariant measure is strictly less than the Hausdorff dimension of the entire fractal, then the set of typical points for the measure is much smaller than the set of atypical points, and therefore the dynamics of ``most'' points on the fractal are not captured by the measure. Even so, we can ask whether the Hausdorff dimension of the fractal can be approximated by the Hausdorff dimensions of invariant measures, i.e. whether it is equal to the supremum of the Hausdorff dimensions of such measures. We call the latter number the \emph{dynamical dimension} of the system; cf. \cite{DenkerUrbanski2}, \cite[\6\612.2-12.3]{PrzytyckiUrbanski}, though we note that the definition of the dynamical dimension in these references is slightly different from ours.

The question of which dynamical systems have ergodic invariant measures of full Hausdorff dimension has generated substantial interest over the past few decades, see e.g. \cite{Baranski2, Barreira, Feng, FengHu, FriedlandOchs, Hutchinson, Kaenmaki2, KotusUrbanski4, Luzia, MauldinUrbanski1, MauldinUrbanski2, McCluskeyManning, PrzytyckiRivera, Reeve, Urbanski6, Yayama}, as well as the survey articles \cite{BarreiraGelfert, ChenPesin, GatzourasPeres_survey, SchmelingWeiss} and the books \cite{Barreira_book1, Barreira_book2}. Most of the results are positive, proving the existence and uniqueness of a measure of full dimension under appropriate hypotheses on the dynamical system. 

The theory in the case of (compact) expanding systems that are conformal or essentially one-dimensional is, in a sense, the most complete -- the Hausdorff and box dimensions of the repeller coincide, and there exists a unique ergodic invariant full dimension measure. The equality of dimension characteristics as well as the existence of a full dimension measure is a consequence of \emph{Bowen's formula} in the thermodynamic formalism, which equates the Hausdorff dimension of the repeller with the unique zero of a pressure functional, see e.g. \cite[Corollary 9.1.7]{PrzytyckiUrbanski}, \cite{GatzourasPeres2}, or \cite[Theorem 2.1]{Rugh} for an elementary proof. The uniqueness of the full dimension measure follows from the \emph{Volume Lemma}, which describes how to compute the Hausdorff dimension of an arbitrary ergodic invariant measure, see e.g. \cite[Theorems 9.1.11 and 9.3.1]{PrzytyckiUrbanski}. On the other hand, if either of the assumptions of compactness and expansion is dropped, then a full dimension measure may not exist, see \cite{UrbanskiZdunik3} and \cite{AvilaLyubich2} respectively.

Another class of examples for which a great deal of theory has been established is the case of two-dimensional Axiom A diffeomorphisms. Loosely speaking, Axiom A diffeomorphisms are those in which there is a dichotomy between ``expanding'' directions and ``contracting'' directions, see e.g. \cite{Bowen_book} for a beautiful introduction. McCluskey and Manning \cite{McCluskeyManning} showed that ``most'' two-dimensional Axiom A diffeomorphisms have basic sets whose Hausdorff dimension is strictly greater than their dynamical dimension (i.e. the supremal dimension of invariant measures), and in particular there are no invariant measures of full dimension. So in the (topologically) generic case there can be no theory of full dimension measures. There is also a simple sufficient condition (not satisfied generically) for the existence of full dimension measures for two-dimensional Axiom A diffeomorphisms, see \cite[Theorem 1.10]{FriedlandOchs}. This condition is also necessary, at least in the case where the system is topologically conjugate to a topologically mixing shift space, as can be seen by combining \cite[p.99]{BarreiraWolf} with \cite[Theorem 1.28]{Bowen_book}.

Progress beyond these cases, and in particular in the case where the system is expanding but may have different rates of expansion in different directions, has been much slower and of more limited scope, see e.g. \cite{BarreiraGelfert, ChenPesin, GatzourasPeres_survey, SchmelingWeiss}. Such systems, called ``expanding repellers'', form another large and much-studied class of examples. They can be formally defined as follows:

\begin{definition}
An \emph{expanding repeller} is a dynamical system $f:K \to K$, where $K$ is a compact subset of a Riemannian manifold $M$, $U \subset M$ is a neighborhood of $K$, and $f:U\to M$ is a $C^1$ transformation such that
\begin{itemize}
\item $f^{-1}(K) = K$; and
\item for some $n$, $f^n$ is infinitesimally expanding on $K$ with respect to the Riemannian metric.
\end{itemize}
\end{definition}

The following question regarding such systems, stated by Schmeling and Weiss to be ``one of the major open problems in the dimension theory of dynamical systems'' \cite[p.440]{SchmelingWeiss}, dates back to at least the early 1990s and can be found reiterated in several places in the literature by various experts in the field (see {Lalley--Gatzouras (1992) \cite[p.4]{LalleyGatzouras}}, {Kenyon--Peres (1996) \cite[Open Problem]{KenyonPeres}}, {Gatzouras--Peres (1996) \cite[Problem 1]{GatzourasPeres_survey}}, {Gatzouras--Peres (1997) \cite[Conjecture on p.166]{GatzourasPeres2}}, {Peres--Solomyak (2000) \cite[Question 5.1]{PeresSolomyak}}, {Schmeling--Weiss (2001) \cite[p.440]{SchmelingWeiss}}, {Petersen (2002) \cite[p.188]{Petersen2}}, {Chen--Pesin (2010) \cite[p.R108]{ChenPesin}}, {Schmeling (2012) \cite[p.298]{Schmeling}}, {Barreira (2013) \cite[p.5]{Barreira_book2}}): 

\begin{question}
\label{mainquestion}
{\it Does every expanding repeller have an ergodic invariant measure of full dimension?}
\end{question}

In this paper we will prove that the answer to Question \ref{mainquestion} is negative by constructing a piecewise affine expanding repeller topologically conjugate to the full shift whose Hausdorff dimension is strictly greater than its dynamical dimension. This expanding repeller will belong to a class of sets that we call ``self-affine sponges'' (not all of which are expanding repellers), and we develop tools for calculating the Hausdorff and dynamical dimensions of self-affine sponges more generally. This makes our paper an extension of several known results about self-affine sponges \cite{Bedford, McMullen_carpets, LalleyGatzouras, KenyonPeres, Baranski}, though in all previously studied cases, the Hausdorff and dynamical dimensions have turned out to be equal. We also note that self-affine sponges are a subclass of the more general class of self-affine sets, and that it is known that almost every self-affine set (with respect to a certain measure on the space of perturbations of a given self-affine set) has an ergodic invariant measure of full dimension \cite{Kaenmaki2}. However, self-affine sponges do not represent typical instances of self-affine sets and so this result does not contradict our theorems. Nevertheless, we show that our counterexamples represent a non-negligible set of self-affine sponges (in the sense of containing a nonempty open subset of the parameter space); see Theorem \ref{theoremcontinuous}.

Previous approaches to Question \ref{mainquestion} have involved using the thermodynamic formalism to compute the Hausdorff dimension of the repeller and then comparing with the dimensions of the invariant measures calculated using the Volume Lemma or its generalization, the Ledrappier--Young dimension formula \cite[Corollary D$'$]{LedrappierYoung2}. When it works, this strategy generally shows that the Hausdorff and dynamical dimensions of a repeller are equal. By contrast, we still use the Ledrappier--Young formula to calculate the dimension of invariant measures, but our strategy to calculate the dimension of the repeller is to pay more attention to the {\it non-invariant} measures. Indeed, we write the Hausdorff dimension of a self-affine sponge as the supremum of the Hausdorff dimensions of certain particularly nice non-invariant measures that we call ``pseudo-Bernoulli'' measures (see Definition \ref{definitionPB}), which are relatively homogeneous with respect to space, but whose behavior with respect to length scale varies in a periodic way. The dimension of these measures turns out to be calculable via an appropriate analogue of the Ledrappier--Young formula, which is how we show that it is sometimes larger than the dimension of any invariant measure.\\

{\bf Acknowledgements.} The first-named author was supported in part by a 2016-2017 Faculty Research Grant from the University of Wisconsin--La Crosse. The second-named author was supported by the EPSRC Programme Grant EP/J018260/1. The authors thank Antti K\"aenm\"aki for helpful comments. The authors also thank an anonymous referee for a very thorough report, which made a number of useful suggestions and detailed comments to help us improve the precision and readability of the paper. 

\section{Main results}
\label{sectionmain}

\subsection{Qualitative results}
\label{subsectionqualitative}

\begin{definition}
\label{definitionsponges}
Fix $d \geq 1$, and let $D = \{1,\ldots,d\}$. For each $i \in D$, let $A_i$ be a finite index set, and let $\Phi_i = (\phi_{i,a})_{a\in A_i}$ be a finite collection of contracting similarities of $[0,1]$, called the \emph{base IFS in coordinate $i$}. (Here IFS is short for \emph{iterated function system}.) Let $A = \prod_{i\in D} A_i$, and for each $\aa = (a_1,\ldots,a_d) \in A$, consider the contracting affine map $\phi_\aa : [0,1]^d \to [0,1]^d$ defined by the formula
\[
\phi_\aa(x_1,\ldots,x_d) = (\phi_{\aa,1}(x_1),\ldots,\phi_{\aa,d}(x_d)),
\]
where $\phi_{\aa,i}$ is shorthand for $\phi_{i,a_i}$ in the formula above, as well as elsewhere. Geometrically, $\phi_\aa$ can be thought of as corresponding to the rectangle 
\[
\phi_\aa([0,1]^d) = \prod_{i\in D} \phi_{\aa,i}([0,1]) \subset [0,1]^d .
\]
Given $E \subset A$, we call the collection $\Phi \df (\phi_\aa)_{\aa\in E}$ a \emph{diagonal IFS}. The \emph{coding map} of $\Phi$ is the map $\pi:E^\N \to [0,1]^d$ defined by the formula
\[
\pi(\omega) = \lim_{n\to\infty} \phi_{\omega\given n}(\0),
\]
where $\phi_{\omega\given n} \df \phi_{\omega_1}\circ\cdots\circ\phi_{\omega_n}$. Finally, the \emph{limit set} of $\Phi$ is the set $\Lambda_\Phi \df \pi(E^\N)$. We call the limit set of a diagonal IFS a \emph{self-affine sponge}. It is a special case of the more general notion of an \emph{self-affine set}, see e.g. \cite{Falconer4}.
\end{definition}

\begin{remark*}
This definition excludes some sets that it is also natural to call ``sponges'', namely the limit sets of affine iterated function systems whose contractions preserve the class of coordinate-parallel rectangles, see e.g. \cite{Fraser2}. The linear parts of such contractions are matrices that can be written as the composition of a permutation matrix and a diagonal matrix. Self-affine sets resulting from these ``coordinate-permuting IFSes'' are significantly more technical to deal with, so for simplicity we restrict ourselves to the case of sponges coming from diagonal IFSes.
\end{remark*}

When $d = 2$, self-affine sponges are called \emph{self-affine carpets}, and have been studied in detail. Their Hausdorff dimensions were computed by Bedford \cite{Bedford}, McMullen \cite{McMullen_carpets}, Lalley--Gatzouras \cite{LalleyGatzouras}, and Bara\'nski \cite{Baranski}, assuming that various conditions are satisfied. Since we will be interested in the higher-dimensional versions of these conditions, we define them now:

\begin{figure}
\scalebox{0.45}{
\begin{tikzpicture}[line cap=round,line join=rounr]
\clip(-1,-1) rectangle (7,7);
\fill[gray] (0,0) -- (0,3) -- (2,3) -- (2,0) -- cycle;
\fill[gray] (2,3) -- (2,6) -- (4,6) -- (4,3) -- cycle;
\fill[gray] (4,0) -- (4,3) -- (6,3) -- (6,0) -- cycle;
\draw (0,0)-- (0,6);
\draw (2,0)-- (2,6);
\draw (4,0)-- (4,6);
\draw (6,0)-- (6,6);
\draw (0,0)-- (6,0);
\draw (0,3)-- (6,3);
\draw (0,6)-- (6,6);
\end{tikzpicture}}
\scalebox{0.45}{
\begin{tikzpicture}[line cap=round,line join=rounr]
\clip(-1,-1) rectangle (7,7);
\fill[gray] (0,0) -- (0,2) -- (3,2) -- (3,0) -- cycle;
\fill[gray] (1,5) -- (1,6) -- (4,6) -- (4,5) -- cycle;
\fill[gray] (2,3) -- (2,4) -- (6,4) -- (6,3) -- cycle;
\draw (0,0)-- (0,6);
\draw (6,0)-- (6,6);
\draw (0,0)-- (6,0);
\draw (0,6)-- (6,6);
\draw (0,0) -- (0,2) -- (3,2) -- (3,0) -- cycle;
\draw (1,5) -- (1,6) -- (4,6) -- (4,5) -- cycle;
\draw (2,3) -- (2,4) -- (6,4) -- (6,3) -- cycle;
\end{tikzpicture}}
\scalebox{0.45}{
\begin{tikzpicture}[line cap=round,line join=rounr]
\clip(-1,-1) rectangle (7,7);
\fill[gray] (0,0) -- (0,2) -- (3,2) -- (3,0) -- cycle;
\fill[gray] (0,4) -- (0,5.5) -- (3,5.5) -- (3,4) -- cycle;
\fill[gray] (4,1) -- (4,2) -- (6,2) -- (6,1) -- cycle;
\fill[gray] (4,5) -- (4,6) -- (6,6) -- (6,5) -- cycle;
\fill[gray] (3,3) -- (3,3.5) -- (4,3.5) -- (4,3) -- cycle;
\draw (0,0)-- (0,6);
\draw (6,0)-- (6,6);
\draw (0,0)-- (6,0);
\draw (0,6)-- (6,6);
\draw (0,0) -- (0,2) -- (3,2) -- (3,0) -- cycle;
\draw (0,4) -- (0,5.5) -- (3,5.5) -- (3,4) -- cycle;
\draw (4,1) -- (4,2) -- (6,2) -- (6,1) -- cycle;
\draw (4,5) -- (4,6) -- (6,6) -- (6,5) -- cycle;
\draw (3,3) -- (3,3.5) -- (4,3.5) -- (4,3) -- cycle;
\draw (3,0)-- (3,6);
\draw (4,0)-- (4,6);
\end{tikzpicture}}
\scalebox{0.45}{
\begin{tikzpicture}[line cap=round,line join=rounr]
\clip(-1,-1) rectangle (7,7);
\fill[gray] (0,0) -- (0,3) -- (3,3) -- (3,0) -- cycle;
\fill[gray] (3,3) -- (3,5) -- (4,5) -- (4,3) -- cycle;
\fill[gray] (4,0) -- (4,3) -- (6,3) -- (6,0) -- cycle;
\fill[gray] (0,5) -- (0,6) -- (3,6) -- (3,5) -- cycle;
\draw (0,0)-- (0,6);
\draw (3,0)-- (3,6);
\draw (4,0)-- (4,6);
\draw (6,0)-- (6,6);
\draw (0,0)-- (6,0);
\draw (0,3)-- (6,3);
\draw (0,5)-- (6,5);
\draw (0,6)-- (6,6);
\end{tikzpicture}}
\caption{Generating templates for a Sierpi\'nski carpet (left), carpets satisfying the coordinate ordering condition (two middle pictures), and a Bara\'nski carpet (right). Each picture defines a diagonal IFS: each shaded region corresponds to an affine contraction that sends the entire unit square to that shaded region. The right middle picture satisfies an additional disjointness condition which makes it a \emph{Lalley--Gatzouras carpet}; cf. Definition \ref{definitionLG}.}
\label{figureBMLGB}
\end{figure}
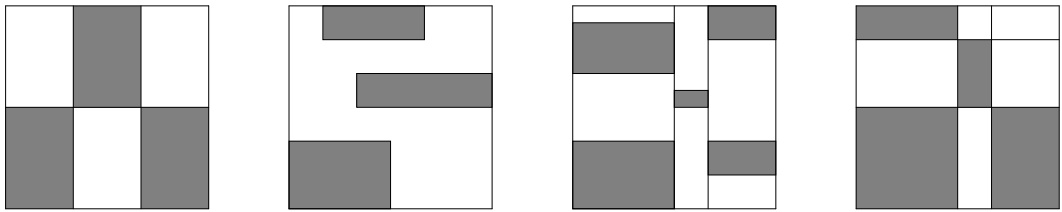
\begin{definition}[Cf. Figure \ref{figureBMLGB}]
\label{definitionBMLGB}
Let $\Lambda_\Phi$ be a self-affine sponge defined by a diagonal IFS $\Phi$.
\begin{itemize}
\item We say that $\Phi$ or $\Lambda_\Phi$ is \emph{Sierpi\'nski} if the base IFSes are of the form
\begin{align*}
\Phi_i &= (\phi_{i,a})_{0 \leq a \leq m_i - 1},&
\phi_{i,a}(x) &= \frac{a + x}{m_i}
\end{align*}
for some distinct integers $m_1,\ldots,m_d \geq 2$.
\item We say that $\Phi$ or $\Lambda_\Phi$ satisfies the \emph{coordinate ordering condition} if there exists a permutation $\sigma$ of $D$ such that for all $\aa\in E$, we have
\[
|\phi_{\aa,\sigma(1)}'| > \cdots > |\phi_{\aa,\sigma(d)}'|.
\]
\item We say that $\Phi$ or $\Lambda_\Phi$ is \emph{Bara\'nski} (resp. \emph{strongly Bara\'nski}) if the base IFSes all satisfy the open set condition (resp. the strong separation condition) with respect to the interval $\I = (0,1)$ (resp. $\I = [0,1]$), i.e. for all $i\in D$, the collection
\[
\big(\phi_{i,a}(\I)\big)_{a\in A_i}
\]
is disjoint.
\end{itemize}
\end{definition}
Notice that every Sierpi\'nski sponge satisfies the coordinate ordering condition and is also Bara\'nski. 
Bedford \cite{Bedford} and McMullen \cite{McMullen_carpets} independently computed the Hausdorff dimension of Sierpi\'nski carpets, and consequently these carpets are sometimes known as \emph{Bedford--McMullen carpets}. Bara\'nski computed the Hausdorff dimension of what we call Bara\'nski carpets \cite{Baranski}.\Footnote{Read literally, the setup of \cite{Baranski} implies that the maps $\phi_{i,a}$ ($i\in D$, $a\in A_i$) are orientation-preserving, but there is no significant difference in dealing with the case where reflections are allowed.}  On the other hand, the coordinate ordering condition, which can be thought of as guaranteeing a ``clear separation of Lyapunov directions'', cf. \cite[p.643]{BarreiraGelfert}, is a higher-dimensional generalization of one of the assumptions of Lalley--Gatzouras \cite{LalleyGatzouras}. Their other assumption is a disjointness condition \cite[p.534]{LalleyGatzouras} that is slightly weaker than the Bara\'nski condition. The higher-dimensional analogue of the disjointness condition is somewhat technical to state, so we defer its definition until Section \ref{sectionweaker}.

\begin{observation}
\label{observationSSC}
Let $\Lambda_\Phi$ be a strongly Bara\'nski sponge. Then the coding map $\pi:E^\N \to \Lambda_\Phi$ is a homeomorphism. It follows that there is a unique map $f:\Lambda_\Phi \to \Lambda_\Phi$ such that $f\circ \pi = \pi \circ \sigma$, where $\sigma:E^\N \to E^\N$ is the shift map. In fact, the dynamical system $f:\Lambda_\Phi \to \Lambda_\Phi$ is a piecewise affine expanding repeller: for all $\aa\in E$, we have $f = \phi_\aa^{-1}$ on $\phi_\aa(\Lambda_\Phi)$.
\end{observation}

In \cite{Bedford,McMullen_carpets,LalleyGatzouras,Baranski}, a relation was established between the Hausdorff dimension of a self-affine carpet $\Lambda_\Phi$ and the Hausdorff dimension of the Bernoulli measures on $\Lambda_\Phi$. Here, a \emph{Bernoulli measure} is a measure of the form
\[
\nu_\pp = \pi_*[\pp^\N],
\]
where $\pp$ is a probability measure on $E$, and $\pi_*[\mu]$ denotes the pushforward of a measure $\mu$ under the coding map $\pi$. In what follows, we let $\PP$ denote the space of probability measures on $E$.

\begin{theorem}[\cite{Baranski}, special cases \cite{Bedford,McMullen_carpets,LalleyGatzouras}]
\label{theoremLGB}
Let $\Lambda_\Phi$ be a Bara\'nski carpet (i.e. a two-dimensional Bara\'nski sponge). Then the Hausdorff dimension of $\Lambda_\Phi$ is equal to the supremum of the Hausdorff dimensions of the Bernoulli measures on $\Lambda_\Phi$, i.e.
\begin{equation}
\label{equality1}
\HD(\Phi) = \sup_{\pp\in\PP} \HD(\nu_\pp),
\end{equation}
where $\HD(\Phi)$ denotes the Hausdorff dimension of $\Lambda_\Phi$.
\end{theorem}

It is natural to ask whether Theorem \ref{theoremLGB} can be generalized to higher dimensions. This question was answered by Kenyon and Peres \cite{KenyonPeres} in the case of Sierpi\'nski sponges:

\begin{theorem}[{\cite[Theorem 1.2]{KenyonPeres}}, special cases \cite{Bedford,McMullen_carpets}]
\label{theoremBMhigher}
The formula \eqref{equality1} holds for Sierpi\'nski sponges (in all dimensions).
\end{theorem}

These results might lead one to conjecture that the formula \eqref{equality1} holds for all Bara\'nski sponges, or at least all Bara\'nski sponges satisfying the coordinate ordering condition. If that fails, one might still conjecture that the Hausdorff dimension of a Bara\'nski sponge is attained by some ergodic invariant measure, even if that measure is not a Bernoulli measure. For example, Neunh\"auserer showed that the formula \eqref{equality1} fails for a certain class of non-Bara\'nski self-affine carpets \cite[Theorem 2.2]{Neunhauserer}, but later it was shown that these carpets do in fact have ergodic invariant measures of full dimension \cite[Theorem 2.15]{FengHu}. Similar examples appear in the realms of conformal iterated function systems satisfying the open set condition \cite{Hutchinson,MauldinUrbanski1,MauldinUrbanski2}, affine iterated function systems with randomized translational parts \cite{Barreira,Kaenmaki2}, and certain non-conformal non-affine iterated function systems \cite{Reeve}, though in these settings, it was not expected that the measure of full dimension would be a Bernoulli measure. This leads to the following definition:

\begin{definition}
\label{definitionDD}
The \emph{dynamical dimension} of a self-affine sponge $\Lambda_\Phi$ is the number
\[
\DynD(\Phi) \df \sup_\mu\{\HD(\pi_*[\mu])\},
\]
where the supremum is taken over all probability measures $\mu$ on $E^\N$ that are invariant under the shift map.
\end{definition}

It turns out that this definition does not help at getting larger dimensions:

\begin{theorem}
\label{theoremDDbernoulli}
The dynamical dimension of a Bara\'nski sponge $\Lambda_\Phi$ is equal to the supremum of the Hausdorff dimensions of its Bernoulli measures, i.e.
\begin{equation}
\label{DDbernoulli}
\DynD(\Phi) = \sup_{\pp\in\PP} \HD(\nu_\pp).
\end{equation}
\end{theorem}

The question remains whether the dynamical dimension is equal to the Hausdorff dimension of $\Lambda_\Phi$. It follows directly from the definition that
\[
\HD(\Phi) \geq \DynD(\Phi).
\]
The main result of this paper is that this inequality is sometimes strict:

\begin{theorem}[Existence of sponges with a dimension gap]
\label{theoremcounterexample}
For all $d \geq 3$, there exists a strongly Bara\'nski sponge $\Lambda_\Phi \subset [0,1]^d$ satisfying the coordinate ordering condition such that
\[
\HD(\Phi) > \DynD(\Phi).
\]
\end{theorem}

Since the sponge $\Lambda_\Phi$ appearing in this theorem is strongly Bara\'nski, there exists a piecewise affine expanding repeller $f:\Lambda_\Phi \to \Lambda_\Phi$ such that $f\circ \pi = \pi\circ\sigma$, where $\sigma : E^\N \to E^\N$ is the shift map (cf. Observation \ref{observationSSC}). Thus, Theorem \ref{theoremcounterexample} shows that the answer to Question \ref{mainquestion} is negative.

The contrast between Theorems \ref{theoremLGB} and \ref{theoremcounterexample} shows that the behavior of self-affine sponges is radically different in the two-dimensional and three-dimensional settings. See Remark \ref{remark2D3D} for some ideas about the cause of this difference.

A natural follow-up question is how common sponges with a dimension gap are. One way to measure this is to ask whether they represent a positive measure subset of the parameter space. We answer this question affirmatively by showing that dimension gaps are stable under perturbations: any Bara\'nski sponge whose defining IFS is sufficiently close to the defining IFS of a Bara\'nski sponge with a dimension gap also has a dimension gap. Equivalently, the class of Bara\'nski IFSes whose limit sets have a dimension gap is an open subset of the parameter space. This is an immediate corollary of the following theorem:

\begin{theorem}
\label{theoremcontinuous}
The functions
\begin{align}
\label{HDDD}
\Phi &\mapsto \HD(\Phi),&
\Phi &\mapsto \DynD(\Phi)
\end{align}
are continuous on the space of Bara\'nski IFSes.
\end{theorem}

\begin{remark*}
It is not too hard to modify the proof of Theorem \ref{theoremcontinuous} to get a stronger result: the functions \eqref{HDDD} are computable in the sense of computable analysis (see \cite{Weihrauch} for an introduction). This means that there is an algorithm that outputs arbitrarily accurate approximations of $\HD(\Phi)$ and $\DynD(\Phi)$, given as input a sequence of approximations of $\Phi$. Every computable function is continuous \cite[Theorem 4.3.1]{Weihrauch}; the converse is not true, since there are only countably many computable functions.
\end{remark*}

\subsection{Computational results}
\label{subsectioncomputation}
The strategy of the proof of Theorem \ref{theoremcounterexample} is to come up with general formulas for the Hausdorff and dynamical dimensions of a Bara\'nski sponge, and then to compare them in a concrete example. For example, Theorem \ref{theoremDDbernoulli} gives a way to compute the dynamical dimension once the dimensions of the Bernoulli measures are known. To get a similar result for the Hausdorff dimension, we introduce a new class of measures which we call ``pseudo-Bernoulli''. These measures are not invariant, since if they were then their dimension could be no bigger than the dynamical dimension.

\begin{definition}
\label{definitionPB}
Recall that $\PP$ denotes the space of probability measures on $E$, the alphabet of the IFS. Given $\lambda > 1$, we call a function $\rr:(0,\infty)\to\PP$ \emph{exponentially $\lambda$-periodic} if for all $b > 0$, we have $\rr_{\lambda b} = \rr_b$. Here we denote the value of $\rr$ at the argument $b$ by $\rr_b$ instead of $\rr(b)$. We call $\rr$ \emph{exponentially $1$-periodic} if it is constant. (The advantange of this definition is that the uniform limit of exponentially $\lambda$-periodic continuous functions as $\lambda \searrow 1$ is exponentially $1$-periodic.) The class of exponentially $\lambda$-periodic continuous functions will be denoted $\RR_\lambda$, and the union will be denoted $\RR = \bigcup_{\lambda \geq 1} \RR_\lambda$. Elements of $\RR$ will be called \emph{cycles on $E$}. Finally, a \emph{pseudo-Bernoulli} measure is a measure of the form $\nu_\rr \df \pi_*[\mu_\rr]$, where $\rr\in\RR$, and
\begin{equation}
\label{PBdef}
\mu_\rr \df \prod_{n\in\N} \rr_n
\end{equation}
is a probability measure on $E^\N$.
\end{definition}

The following theorem subsumes Theorems \ref{theoremLGB} and \ref{theoremBMhigher} as special cases, see Section \ref{sectionequality} for details. The techniques we use to prove it are similar to the techniques originally used to prove Theorems \ref{theoremLGB} and \ref{theoremBMhigher}.

\begin{theorem}
\label{theoremHDcompute}
The Hausdorff dimension of a Bara\'nski sponge $\Lambda_\Phi$ is equal to the supremum of the Hausdorff dimensions of its pseudo-Bernoulli measures, i.e.
\begin{equation}
\label{HDcompute}
\HD(\Phi) = \sup_{\rr\in\RR} \HD(\nu_\rr).
\end{equation}
\end{theorem}

\begin{remark*}
The inequality $\HD(\Phi) \geq \sup_{\rr\in\RR} \HD(\nu_\rr)$, which forms the easy direction of Theorem \ref{theoremHDcompute}, is all that is needed in the proof of Theorem \ref{theoremcounterexample}. However, the proof of Theorem \ref{theoremHDcompute} provides some motivation for why it is appropriate to consider measures of the form $\nu_\rr$ in the proof of Theorem \ref{theoremcounterexample}. Indeed, proving Theorem \ref{theoremHDcompute} is what caused the authors to start paying attention to the class of pseudo-Bernoulli measures.
\end{remark*}

Of course, Theorem \ref{theoremHDcompute} raises the question of how to compute the Hausdorff dimension of a pseudo-Bernoulli measure $\nu_\rr$. Similarly, Theorem \ref{theoremDDbernoulli} raises the (easier) question of how to compute the Hausdorff dimension of a Bernoulli measure $\nu_\pp$ -- which is answered by a Ledrappier--Young type formula (cf. \eqref{deltap2}). In fact, the latter question can be viewed as a special case of the former, since every Bernoulli measure is also a pseudo-Bernoulli measure. As a matter of notation, if $\pp \in \PP$, then we let $\pp$ also denote the constant cycle $b\mapsto\pp_b = \pp$, so that we can think of $\PP$ as being equal to $\RR_1 \subset \RR$. Note that the notation $\nu_\pp$ means the same thing whether we interpret it as referring to the Bernoulli measure corresponding to $\pp\in \PP$, or the pseudo-Bernoulli measure corresponding to the constant cycle $\pp\in \RR_1$.

To compute the Hausdorff dimension of pseudo-Bernoulli measures, we need to introduce some more notation and definitions:

\begin{notation}
\label{notationRB}
For each $\rr\in\RR$ and $B > 0$, we let
\begin{align}
\label{RBdef}
\bfR_B &= \int_0^B \rr_b\;\dee b,&
\what\bfR_B &= B^{-1} \bfR_B \in \PP.
\end{align}
Note that if $\rr$ is exponentially $\lambda$-periodic, then so is $\what\bfR$. We will use a similar convention with other letters in place of $\rr$; for example, if $\pp\in\PP$ then we write $\bfP_B = \int_0^B \pp_b \;\dee b = \int_0^B \pp \;\dee b = B\pp$ and $\what\bfP_B = B^{-1} \bfP_B = \pp$.
\end{notation}

\begin{definition}
\label{definitionlyapent}
Given $\pp\in\PP$ and $i\in D$, the \emph{$i$th Lyapunov exponent\Footnote{This terminology is not meant to imply that the Lyapunov exponents are distinct or have been arranged in increasing order, although it is often convenient to assume the latter (cf. Proposition \ref{propositionXinc} below).} of $\pp$} is the number
\[
\chi_i(\pp) \df -\int \log|\phi_{\aa,i}'| \;\dee\pp(\aa).
\]
Note that this definition makes sense even if the total mass of $\pp$ is not $1$, and we will use it sometimes in this more general sense. Given a coordinate set $I \subset D$, the \emph{entropy of $I$ with respect to $\pp$} is the number
\[
h_I(\pp) = h(I;\pp) \df -\int \log\pp([\aa]_I) \;\dee\pp(\aa),
\]
where
\begin{equation}
\label{aIdef}
[\aa]_I = \{\bb\in E : a_i = b_i \all i \in I\}.
\end{equation}
Note that $[\aa]_D = \{\aa\}$ and $[\aa]_\smallemptyset = E$.

Finally, given $I \subset I' \subset D$, then \emph{conditional entropy of $I'$ relative to $I$ with respect to $\pp$} is the number
\[
h(I'\given I;\pp) \df h(I';\pp) - h(I;\pp) = \int \log\frac{\pp([\aa]_I)}{\pp([\aa]_{I'})} \;\dee\pp(\aa).
\]
\end{definition}

\begin{definition}
\label{definitionnondegenerate}
Given $\rr\in\RR$, we let $E_\rr = \{\aa\in E : \rr_b(\aa) > 0 \text{ for some $b > 0$}\}$. We say that $\rr$ is \emph{nondegenerate} if the set $\{b > 0 : \rr_b(\aa) > 0 \text{ for all } \aa\in E_\rr\}$ is dense in $(0,\infty)$, and we denote the space of nondegenerate cycles by $\RR^*$. We also write $\RR_\lambda^* = \RR_\lambda\cap\RR^*$.

Note that every measure is nondegenerate when considered as a constant cycle.
\end{definition}

\begin{theorem}
\label{theoremPBcompute}
Let $\Lambda_\Phi$ be a Bara\'nski sponge. Then for all $\lambda \geq 1$ and $\rr\in\RR_\lambda^*$, the dimension $\HD(\nu_\rr)$ can be computed by the formula
\begin{equation}
\label{deltar}
\HD(\nu_\rr) = \delta(\rr) \df \inf_{B \in [1,\lambda]} \delta(\rr,B),
\end{equation}
where for each $B > 0$,
\begin{equation}
\label{deltarB1}
\delta(\rr,B) \df \frac{1}{B} \int_0^\infty h(\{i\in D : b \leq B_i\};\rr_b) \;\dee b,
\end{equation}
where the numbers $B_1,\ldots,B_d > 0$ are chosen so that
\begin{equation}
\label{Bidef}
B = \int_0^{B_i} \chi_i(\rr_b) \;\dee b = \chi_i(\bfR_{B_i}).
\end{equation}
If $\rr\in\RR_\lambda\butnot\RR_\lambda^*$, then $\HD(\nu_\rr) \leq \delta(\rr)$. The terms $\HD(\nu_\rr)$ $(\rr\in\RR_\lambda\butnot\RR_\lambda^*)$ do not contribute to the supremum in \eqref{HDcompute}.

In particular, for all $\pp\in\PP$, the dimension $\HD(\nu_\pp)$ can be computed by the formula
\begin{equation}
\label{deltap1}
\HD(\nu_\pp) = \delta(\pp) \df \int_0^\infty h(\{i\in D : b \leq 1/\chi_i(\pp)\};\pp) \;\dee b.
\end{equation}
\end{theorem}

We remark that the map $B \mapsto \delta(\rr,B)$ is exponentially $\lambda$-periodic, so that the infimum in \eqref{deltar} would be the same if it was taken over all $B > 0$ rather than only over $B \in [1,\lambda]$. We also remark on the geometric meaning of the quantities $B_1,\ldots,B_d$: if $\omega\in E^\N$ is a $\mu_\rr$-typical point and $\rho = e^{-B}$, then $B_i$ is approximately the number of coordinates of $\omega$ that must be known before the $i$th coordinate of $\pi(\omega)$ can be computed with accuracy $\rho$. Thus the numbers $B_1,\ldots,B_d$ are useful at estimating the $\nu_\rr$-measure of the ball $B(\pi(\omega),\rho)$. For a more rigorous presentation of this idea, see the proof of Theorem \ref{theoremPBcompute}.

Formulas \eqref{deltarB1} and \eqref{deltap1} share a particularly nice feature, viz. their validity does not depend on the ordering of the numbers $B_1,\ldots,B_d$ (in the case of \eqref{deltarB1}) or of the Lyapunov exponents $\chi_1(\pp),\ldots,\chi_d(\pp)$ (in the case of \eqref{deltap1}). However, it is sometimes more useful to have versions of these formulas that do depend on the orderings of these numbers. For convenience, for all $i = 0,\ldots,d$ we write
\[
I_{\leq i} = \{1,\ldots,i\},
\]
so that in particular $I_{\leq 0} = \emptyset$ and $I_{\leq d} = D$.

\begin{proposition}
\label{propositionXinc}
If $B_1 \geq \cdots \geq B_d$ for some $\rr\in\RR$ and $B > 0$, then
\begin{equation}
\label{deltarB2}
\delta(\rr,B) = \sum_{i\in D} \frac{\int_0^{B_i} h(I_{\leq i} \given I_{\leq i - 1};\rr_b) \;\dee b}{\int_0^{B_i} \chi_i(\rr_b)\;\dee b}
\leq \sum_{i\in D} \frac{h(I_{\leq i} \given I_{\leq i - 1};\what\bfR_{B_i})}{\chi_i(\what\bfR_{B_i})}\cdot
\end{equation}
In particular, if $\chi_1(\pp) \leq \cdots \leq \chi_d(\pp)$ for some $\pp\in\PP$, then
\begin{equation}
\label{deltap2}
\delta(\pp) = \sum_{i\in D} \frac{h(I_{\leq i} \given I_{\leq i - 1};\pp)}{\chi_i(\pp)}\cdot
\end{equation}
\end{proposition}

\begin{remark*}
The formula \eqref{deltap2} is a special case of a theorem of Feng and Hu \cite[Theorem 2.11]{FengHu}. It can be viewed as an analogue of the well-known Ledrappier--Young formula for the Hausdorff dimension of the unstable leaves of an ergodic invariant measure of a diffeomorphism \cite[Corollary D$'$]{LedrappierYoung2}. In fact, \eqref{deltap2} is close to being a special case of the ``Ledrappier--Young formula for endomorphisms'' \cite[Theorem 2.8 and (19)]{QianXie}, although there are formal difficulties with deducing one from the other.\Footnote{Specifically, it is not clear whether every expanding repeller can be embedded into an expanding global endomorphism of a compact manifold.} Since the formula \eqref{deltarB2} bears some resemblance to \eqref{deltap2}, it can be thought of as extending this Ledrappier--Young-type formula to certain non-invariant measures of a dynamical system.
\end{remark*}

We remark that the results of this section are the first in the literature to address dimension questions regarding self-affine sponges of dimension at least three, with the exception of various results regarding Sierpi\'nski sponges \cite{KenyonPeres, Olsen1, Olsen2}. This significant gap in the literature was recently posed as question by Fraser and Howroyd \cite[Question 4.3]{FraserHowroyd}, namely how to compute the Hausdorff dimension and the upper and lower Assouad and box dimensions of self-affine sponges. The results of this subsection can be seen as partially answering this broad question.\\

{\bf Outline of the paper.}
In Section \ref{sectionweaker} we introduce a weakening of the Bara\'nski assumption that we will use in our proofs. In Section \ref{sectionQBdim} we prove Theorem \ref{theoremPBcompute} and Proposition \ref{propositionXinc}. In Section \ref{sectionHDDD} we prove Theorems \ref{theoremDDbernoulli} and \ref{theoremHDcompute}. In Section \ref{sectioncontinuous} we prove Theorem \ref{theoremcontinuous}. 
In Section \ref{sectionequality} we give new proofs of Theorems \ref{theoremLGB} and \ref{theoremBMhigher} using Theorem \ref{theoremHDcompute}. We prove our main result, Theorem \ref{theoremcounterexample}, in Section \ref{sectioncounterexample}. Finally, in Section \ref{sectionopen} we list a few open questions. The sections are mostly independent of each other, but they are ordered according to the dependencies between the proofs.\\

{\bf Notation.} For the reader's convenience we summarize a list of commonly used symbols below:

\noindent \medskip{}
\vskip3pt
\renewcommand{\arraystretch}{1.5}
\begin{tabular}{ll}
\hline 
IFS & Iterated function system\\
$d$  & Dimension of the ambient Euclidean space\\
$D$ & $D \df \{1,\ldots,d\}$\\
$A_i$ & The alphabet of the base IFS $\Phi_i$\\
$\Phi_i$ & The base IFS in coordinate $i$: $\Phi_i = (\phi_{i,a})_{a\in A_i}$\\
$A$ & The full product alphabet: $A \df \prod_{i\in D} A_i$\\
$E$ & The alphabet of the IFS: $E \subset A$\\
$\Phi$ & The diagonal IFS used to define the self-affine sponge: $\Phi \df (\phi_\aa)_{\aa\in E}$\\
$\pi:E^\N \to [0,1]^d$ & The coding map of $\Phi$\\
$\phi_{\omega\given n}$ & IFS contraction corresponding to the word $\omega\given n$: $\phi_{\omega\given n} \df \phi_{\omega_1}\circ\cdots\circ\phi_{\omega_n}$\\
$\Lambda_\Phi$ & The limit set of $\Phi$: $\Lambda_\Phi \df \pi(E^\N)$\\
$\sigma:E^\N\to E^\N$ & The shift map\\

$\pi_*[\mu]$ & Pushforward of a measure $\mu$ under the coding map $\pi$\\
$\PP$ & The space of probability measures on the alphabet $E$\\
$\nu_\pp$ & Bernoulli measure: $\nu_\pp = \pi_*[\pp^\N]$ for some $\pp \in \PP$\\
$\HD$ & Hausdorff dimension\\
$\DynD$ & Dynamical dimension, see Definition \ref{definitionDD}\\
$\RR$ & Exponentially periodic continuous $\PP$-valued functions, see Definition \ref{definitionPB}\\
$\RR_\lambda$ & Exponentially $\lambda$-periodic continuous $\PP$-valued functions\\
$\QQ$ & Countable dense subset of $\RR$\\
$\rr_b$ & Value of $\rr:(0,\infty)\to\PP$ at $b\in (0,\infty)$\\
$\mu_\rr$ & $\mu_\rr \df \prod_{n\in\N} \rr_n$\\
$\nu_\rr$ & Pseudo-Bernoulli measure: $\nu_\rr \df \pi_*[\mu_\rr]$ for some $\rr\in\RR$\\
$\bfR_B$, $\bfS_B$ etc.\footnotemark & $\bfR_B \df \int_0^B \rr_b\;\dee b$\\
$\what\bfR_B$, $\what\bfS_B$ etc. & $\what\bfR_B \df B^{-1} \bfR_B \in \PP$\\
$\chi_i(\pp)$ & $i$th Lyapunov exponent of $\pp$, see Definition \ref{definitionlyapent}\\
$h_I(\pp) \equiv h(I;\pp)$ & Entropy of $I$ with respect to $\pp$ for a coordinate set $I \subset D$, see Definition \ref{definitionlyapent}\\
$h(I'\given I;\pp)$ & Conditional entropy of $I'$ relative to $I$ with respect to $\pp$, see Definition \ref{definitionlyapent}\\
$[\aa]_I$ & $[\aa]_I \df \{\bb\in E : a_i = b_i \all i \in I\}$\\
$I_{\leq i}$ & $I_{\leq i} \df \{1,\ldots,i\}$\\
\hline
\end{tabular}
\Footnotetext{Expressions such as $\bfS_B$ sometimes appear without a corresponding function $b\mapsto \ss_b\in\PP$, such as in the proof of Theorem \ref{theoremgood2}. However, in these cases the map $B\mapsto \bfS_B$ is still an increasing map from $(0,\infty)$ to the space of measures on $E$ such that $\bfS_B(E) = B$ for all $B > 0$.}

\begin{tabular}{ll}
\hline
$\RR^*$ & Nondegenerate cycles on $E$, see Definition \ref{definitionnondegenerate}\\
$\RR_\lambda^*$ & $\RR_\lambda^* \df \RR_\lambda\cap\RR^*$\\
$\delta(\pp)$ & Formula for computing $\HD(\nu_\pp)$:\\
& \hspace{0.1 in} $\delta(\pp) \df \int_0^\infty h(\{i\in D : b \leq 1/\chi_i(\pp)\};\pp) \;\dee b$\\
$B_i$ & The unique solution to $B = \int_0^{B_i} \chi_i(\rr_b) \;\dee b = \chi_i(\bfR_{B_i})$\\

$\delta(\rr,B)$ & $\delta(\rr,B) \df \frac{1}{B} \int_0^\infty h(\{i\in D : b \leq B_i\};\rr_b) \;\dee b$\\

$\delta(\rr)$ & Formula for computing $\HD(\nu_\rr)$:\\
& \hspace{0.1 in}  $\delta(\rr) \df \inf_{B \in [1,\lambda]} \delta(\rr,B)$, where $\lambda$ is the exponential period of $\rr$\\
$I(\pp,x)$ & $I(\pp,x) \df \{i\in D : \chi_i(\pp) \leq x\}$\\
$\pdim(\xx,\mu)$ & Lower pointwise dimension of $\mu$ at $\xx$\\
$\delta_x$ & Dirac point measure at $x$\\
$X_i(\omega\given N)$ & $X_i(\omega\given N) \df -\log|\phi_{\omega\given N,i}'|$\\
$[\omega\given N]_I$ & $[\omega\given N]_I \df \{\tau\in E^\N : \tau_n \in [\omega_n]_I \all n \leq N\}$\\
$B_\omega(N_1,\ldots,N_d)$ & $B_\omega(N_1,\ldots,N_d) \df \bigcap_{i\in D} [\omega\given {N_i}]_{\{i\}}$\\
$\bfA\cdot\bfB$ & product of matrices $\bfA$ and $\bfB$\\
$\lb \vv,\ww\rb$ & scalar product of vectors $\vv$ and $\ww$\\
$J$ & $J \df \{1,2,3\}$ is the index set for the sub-IFSes of our construction\\
$\Delta$ & Probability measures on $J$\\
$\uu$ & $\uu \df (1/3,1/3,1/3)$\\
$\bfU$ & $\bfU \df [1,1,1]^T\cdot [1,1,1]$\\
\hline
\end{tabular}

\section{Weaker projection conditions}
\label{sectionweaker}

In the theorems of the previous section, we always assumed that the self-affine sponge in question was Bara\'nski -- i.e. that its base IFSes satisfied the open set condition. This assumption is not always necessary and can in some circumstances be replaced by a weaker assumption:

\begin{definition}
\label{definitiongood}
Let $\Lambda_\Phi$ be a self-affine sponge, and let $I \subset D$ be a coordinate set. Let
\[
\Phi_I = (\phi_{I,\aa})_{\aa\in \pi_I(E)},
\]
where $\phi_{I,\aa}:[0,1]^I\to [0,1]^I$ is defined by the formula
\[
\phi_{I,\aa}(\xx) = \big(\phi_{\aa,i}(x_i)\big)_{i\in I}
\]
and $\pi_I:A\to A_I \df \prod_{i\in I} A_i$ is the projection map. We call $I$ \emph{good} if the IFS $\Phi_I$ satisfies the open set condition, i.e. if the collection
\[
\big(\phi_{I,\aa}(\I^I)\big)_{\aa\in \pi_I(E)}
\]
is disjoint, where $\I = (0,1)$. Also, a measure $\pp\in \PP$ is called \emph{good} if for every $x > 0$, the set
\begin{equation}
\label{Ipx}
I(\pp,x) = \{i\in D : \chi_i(\pp) \leq x\}
\end{equation}
is good. Next, a cycle $\rr\in\RR$ is called \emph{good} if the measures $\what\bfR_B$ $(B > 0)$ are all good. Note that $\pp$ is good as a measure if and only if it is good as a constant cycle. Finally, a sponge $\Lambda_\Phi$ is \emph{good} if all measures (and thus also all cycles) on $E$ are good. Note that every Bara\'nski sponge is good, since all of its coordinate sets are good.
\end{definition}

\begin{theorem}[Generalization of Theorem \ref{theoremPBcompute}]
\label{theoremgood1}
Let $\Lambda_\Phi$ be an arbitrary self-affine sponge. Then for all $\rr\in\RR$, we have
\[
\HD(\nu_\rr) \leq \delta(\rr),
\]
with equality if $\rr$ is good and nondegenerate. Here $\delta(\rr)$ is defined in the same way as in Theorem \ref{theoremPBcompute}. In particular, for all $\pp\in\PP$, we have
\[
\HD(\nu_\pp) \leq \delta(\pp),
\]
with equality if $\pp$ is good.
\end{theorem}

\begin{theorem}
\label{theoremgood2}
Let $\Lambda_\Phi$ be an arbitrary self-affine sponge. Then
\begin{align} \label{HDbounds}
\sup_{\substack{\rr\in\RR \\ \text{good}}}\delta(\rr)
\leq \HD(\Phi)
&\leq \sup_{\rr\in\RR}\delta(\rr),\\ \label{DDbounds}
\sup_{\substack{\pp\in\PP \\ \text{good}}} \delta(\pp)
\leq \DynD(\Phi)
&\leq \sup_{\pp\in\PP} \delta(\pp).
\end{align}
\end{theorem}

\begin{corollary}[Generalization of Theorems \ref{theoremDDbernoulli} and \ref{theoremHDcompute}]
\label{corollarygood}
Let $\Lambda_\Phi$ be a good sponge. Then
\begin{align*}
\HD(\Phi)
&= \sup_{\rr\in\RR}\delta(\rr),&
\DynD(\Phi)
&= \sup_{\pp\in\PP} \delta(\pp).
\end{align*}
\end{corollary}

\begin{remark}
In some cases, Theorem \ref{theoremgood2} can still be used to compute the Hausdorff and dynamical dimensions of a sponge $\Lambda_\Phi$ even if that sponge is not good. This is because as long as the supremum of $\delta$ is attained at a good measure (resp. good cycle), then the dynamical (resp. Hausdorff) dimension of $\Lambda_\Phi$ is equal to the dimension of this measure (resp. cycle), regardless of whether or not other measures (resp. cycles) are good.
\end{remark}

Using the terminology of this section, we can also generalize the framework of Lalley and Gatzouras \cite{LalleyGatzouras} to higher dimensions:

\begin{definition}
\label{definitionLG}
A sponge $\Lambda_\Phi$ will be called \emph{Lalley--Gatzouras} if it satisfies the coordinate ordering condition with respect to some permutation $\sigma$ of $D$, such that the sets $\sigma(I_{\leq i})$ ($i\in D$) are all good. Equivalently, a sponge is Lalley--Gatzouras if it is good and satisfies the coordinate ordering condition. 
\end{definition}

We do not prove any theorems specifically about Lalley--Gatzouras sponges, since they do not seem to behave any differently from general good sponges. However, it is worth noting that since all Lalley--Gatzouras sponges are good, all our theorems about good sponges apply to them, so that we are truly generalizing the framework of \cite{LalleyGatzouras} as well as the framework of \cite{Baranski}. We also note that the sponge of Theorem \ref{theoremcounterexample} is a Lalley--Gatzouras sponge, since it is a Bara\'nski sponge that satisfies the coordinate ordering condition.

\section{Dimensions of pseudo-Bernoulli measures}
\label{sectionQBdim}

In this section we compute the Hausdorff dimension of pseudo-Bernoulli measures, proving Theorem \ref{theoremgood1} (which implies Theorem \ref{theoremPBcompute}) and Proposition \ref{propositionXinc}. Our main tool will be the Rogers--Taylor density theorem, a well-known formula for computing the Hausdorff dimension of a measure:

\begin{theorem}[\cite{RogersTaylor}]
\label{theoremRT}
If $\mu$ is a probability measure on $\R^d$ and $S \subset \R^d$ is a set of positive $\mu$-measure, then
\[
\inf_{\xx\in S} \underline\pdim(\xx,\mu)
\leq \HD(S)
\leq \sup_{\xx\in S} \underline\pdim(\xx,\mu),
\]
where
\[
\underline\pdim(\xx,\mu) \df \liminf_{\rho\to 0} \frac{\log\mu(B(\xx,\rho))}{\log(\rho)}
\]
is the lower pointwise dimension of $\mu$ at $\xx$. In particular,
\[
\HD(\mu) = \esssup_{\xx\in\R^d} \underline\pdim(\xx,\mu).
\]
\end{theorem}

We prove Proposition \ref{propositionXinc} first, since it will be used in the proof of Theorem \ref{theoremgood1}. We need a lemma, which will also be used in the proof of Theorem \ref{theoremcounterexample}:

\begin{lemma}[Near-linearity of entropy]
\label{lemmaentropy}
Let $J$ be a finite set, let $(q_j)_{j\in J}$ be a probability vector, and let $(\pp_j)_{j\in J}$ be a family of elements of $\PP$. Then for all $I \subset I' \subset D$,
\begin{equation}
\label{hIlinear}
\sum_{j\in J} q_j h(I'\given I;\pp_j)
\leq h\left(I'\given I ; \sum_{j\in J} q_j \pp_j\right)
\leq \sum_{j\in J} q_j h(I'\given I;\pp_j) + \log\#(J).
\end{equation}
\end{lemma}
\begin{proof}
Let $\pp$ be the probability measure on $J\times E$ given by the formula $\pp = \sum_{j\in J} q_j \delta_j\times\pp_j$, where $\delta_j$ denotes the Dirac point measure at $j$. Consider the partitions on $J\times E$ given by the formulas
\begin{align*}
\AA &\df \{J\times [\aa]_{I'} : \aa\in E\},&
\BB &\df \{J\times [\aa]_I : \aa\in E\},&
\CC &\df \{\{j\}\times E : j\in J\}.
\end{align*}
Then \eqref{hIlinear} is equivalent to the inequalities
\[
H_\pp(\AA \given \BB\vee\CC) \leq H_\pp(\AA\given\BB) \leq H_\pp(\AA\given\BB\vee\CC) + \log\#(\CC),
\]
where $H_\pp(\cdot\given\cdot)$ denotes the standard conditional entropy of two partitions. These inequalities follow from well-known facts about entropy, see e.g. \cite[Theorem 2.3.3(f)]{PrzytyckiUrbanski}.
\end{proof}

\begin{corollary}
\label{corollaryentropy}
Let $J$ be a Borel measurable space, let $\qq$ be a probability measure on $J$, and let $(\pp_j)_{j\in J}$ be a family of elements of $\PP$. Then for all $I \subset I' \subset D$,
\[
\int h(I'\given I;\pp_j) \;\dee \qq(j)
\leq h\left(I'\given I ; \int \pp_j \;\dee \qq(j)\right).
\]
\end{corollary}
\begin{proof}
If $\AA$ is a finite partition of $J$, then Lemma \ref{lemmaentropy} shows that
\[
\sum_{A\in \AA} h\left(I'\given I ; \frac{1}{\qq(A)} \int_A \pp_j \;\dee\qq(j) \right) \qq(A) \leq h\left(I'\given I ; \int \pp_j \;\dee \qq(j)\right).
\]
Letting $\AA$ tend to the partition of $J$ into points completes the proof.
\end{proof}

\begin{proof}[Proof of Proposition \ref{propositionXinc}]
Write $B_{d + 1} = 0$, so that $B_1 \geq \cdots \geq B_{d + 1}$. Then
\[
\{i\in D : b \leq B_i\} = I_{\leq j} \all j = 1,\ldots,d \all b\in (B_{j + 1},B_j),
\]
and $\{i\in D : b \leq B_i\} = \emptyset$ for all $b > B_1$. Thus
\begin{align*}
B\delta(\rr,B)
&= \int h(\{i\in D : b \leq B_i\};\rr_b) \;\dee b\\
&= \sum_{i = 1}^d \int_{B_{i + 1}}^{B_i} h(I_{\leq i};\rr_b)\;\dee b\\
&= \sum_{i = 1}^d \int_0^{B_i} h(I_{\leq i};\rr_b)\;\dee b
- \sum_{i = 2}^{d + 1} \int_0^{B_i} h(I_{\leq i - 1};\rr_b)\;\dee b\\
&= \sum_{i = 1}^d \int_0^{B_i} h(I_{\leq i} \given I_{\leq i - 1};\rr_b)\;\dee b\\
&\leq \sum_{i = 1}^d B_i h(I_{\leq i} \given I_{\leq i - 1};\what\bfR_B). \by{Corollary \ref{corollaryentropy}}
\end{align*}
Dividing by $B$ and then applying \eqref{Bidef} yields \eqref{deltarB2}. Considering the special case where $\rr$ is constant yields \eqref{deltap2}.
\end{proof}

\begin{proof}[Proof of Theorem \ref{theoremgood1}]
For convenience, in this proof we use the max norm on $\R^d$. Fix $\rr\in\RR$, and let $\omega_1,\omega_2,\ldots$ be a sequence of $E$-valued independent random variables, such that the distribution of $\omega_n$ is $\rr_n$. Then $\omega = \omega_1\omega_2\cdots$ is an $E^\N$-valued random variable with distribution $\mu_\rr$. For each $i\in D$, consider the sequence of random variables
\[
\big(-\log|\phi_{\omega_n,i}'|\big)_{n\in\N}
\]
and for each $I\subset D$, consider the sequence of random variables
\[
\big(-\log \rr_n([\omega_n]_I)\big)_{n\in\N}
\]
(cf. \eqref{aIdef}). Each of these sequences is a sequence of independent random variables with uniformly bounded variance,\Footnote{The variance of $-\log\rr_n([\omega_n]_I)$ is at most $\#(E) \max_{x\in [0,1]} x\log^2(x)$.} so by \cite[Corollary A.8]{BenoistQuint_book}\Footnote{This is called Corollary 1.8 in the appendix of the preprint version of \cite{BenoistQuint_book}.} the law of large numbers holds for these sequences, i.e.
\begin{align*}
-\sum_{n = 1}^N \log|\phi_{\omega_n,i}'| &= \sum_{n = 1}^N \chi_i(\rr_n) + o(N)\\
-\sum_{n = 1}^N \log \rr_n([\omega_n]_I) &= \sum_{n = 1}^N h_I(\rr_n) + o(N)
\end{align*}
almost surely. Moreover, since $\rr\in\RR$, we have
\[
\sup_{\substack{b,b' \geq B \\ |b - b'| \leq 1}} \|\rr_{b'} - \rr_b\| \tendsto{B\to\infty} 0,
\]
where $\|\cdot\|$ is any norm on the space of measures of $E$. Since the functions $\chi_i$ ($i\in D$) and $h_I$ ($I\subset D$) are continuous, this implies that
\begin{align*}
\sum_{n = 1}^N \chi_i(\rr_n) &= \int_0^N \chi_i(\rr_b) \;\dee b + o(N)\\
\sum_{n = 1}^N h_I(\rr_n) &= \int_0^N h_I(\rr_b) \;\dee b + o(N).
\end{align*}
Now let us introduce the notation
\begin{align*}
X_i(\omega\given N) &= -\log|\phi_{\omega\given N,i}'|\\
[\omega\given N]_I &= \{\tau\in E^\N : \tau_n \in [\omega_n]_I \all n \leq N\},
\end{align*}
so that
\begin{align*}
X_i(\omega\given N) = -\sum_{n = 1}^N \log|\phi_{\omega_n,i}'| &= \chi_i(\bfR_N) + o(N)\\
-\log \mu_\rr([\omega\given N]_I) = -\sum_{n = 1}^N \log \rr_n([\omega_n]_I) &= \int_0^N h_I(\rr_b) \;\dee b + o(N).
\end{align*}
For all $N_1,\ldots,N_d \in \N$, write
\begin{equation}
\label{Bomegadef}
B_\omega(N_1,\ldots,N_d) \df \bigcap_{i\in D} [\omega\given {N_i}]_{\{i\}},
\end{equation}
and note that
\[
\diam\big(\pi\big(B_\omega(N_1,\ldots,N_d)\big)\big) \leq \max_{i\in D} \exp(-X_i(\omega\given {N_i}))
\]
since we are using the max norm. Now let $\rho > 0$ be a small number, let $B = -\log(\rho)$, and let $B_1,\ldots,B_d > 0$ be given by \eqref{Bidef}. Without loss of generality suppose that $B_1 \geq \cdots \geq B_d$.

We proceed to prove that $\HD(\nu_\rr) \leq \delta(\rr)$. Fix $\epsilon > 0$, and for each $i\in D$ let $N_i = \lfloor (1 + \epsilon) B_i\rfloor$. Then if $B$ is sufficiently large (depending on $\epsilon$), then
\[
X_i(\omega\given {N_i}) \geq \chi_i(\bfR_{B_i}) = B = -\log(\rho) \all i\in D,
\]
and thus
\[
\pi\big(B_\omega(N_1,\ldots,N_d)\big) \subset B(\pi(\omega),\rho).
\]
So
\begin{align*}
-\log \nu_\rr\big(B(\pi(\omega),\rho)\big)
&\leq -\log\mu_\rr\big(B_\omega(N_1,\ldots,N_d)\big)
= -\sum_{n\in\N} \log\rr_n([\omega_n]_{\{i\in D : n\leq N_i\}}) \noreason\\
&= -\sum_{i\in D} \sum_{n = N_{i + 1} + 1}^{N_i} \log\rr_n([\omega_n]_{I_{\leq i}}) \note{with $N_{d + 1} \df 0$}\\
&= \sum_{i\in D} \int_{N_{i + 1}}^{N_i} h(I_{\leq i};\rr_b) \;\dee b + o(N_i)\\
&= \sum_{i\in D} \int_0^{N_i} h(I_{\leq i}\given I_{\leq i - 1};\rr_b) \;\dee b + o(B)\\
&= \sum_{i\in D} \int_0^{B_i} h(I_{\leq i}\given I_{\leq i - 1};\rr_b) \;\dee b + O(\epsilon B) + o(B) \noreason\\
&= B[\delta(\rr,B) + O(\epsilon) + o(1)] \by{Proposition \ref{propositionXinc}}
\end{align*}
and thus
\[
\frac{\log \nu_\rr\big(B(\pi(\omega),\rho)\big)}{\log(\rho)}
\leq \delta(\rr,B) + O(\epsilon) + o(1).
\]
Letting $B\to\infty$ (i.e. $\rho\to 0$) and then $\epsilon \to 0$, we get
\[
\underline\pdim(\pi(\omega),\nu_\rr) \leq \liminf_{B\to \infty} \delta(\rr,B),
\]
where $\underline\pdim$ is as in Theorem \ref{theoremRT}. But since $\rr$ is exponentially periodic, so is $B\mapsto \delta(\rr,B)$, and thus
\[
\liminf_{B\to \infty} \delta(\rr,B) = \inf_{B\in [1,\lambda]} \delta(\rr,B) = \delta(\rr).
\]
Combining with Theorem \ref{theoremRT} proves that $\HD(\nu_\rr) \leq \delta(\rr)$.

Now suppose that $\rr$ is good and nondegenerate, and we will show that $\HD(\nu_\rr) \geq \delta(\rr)$. Without loss of generality assume that $E_\rr = E$. Consider the numbers $N_i \df \lfloor (1 - \epsilon) B_i\rfloor$ ($i\in D$). We will show that
\begin{equation}
\label{ETSQB}
\pi^{-1}\big(B(\pi(\omega),\rho)\big) \subset B_\omega(N_1,\ldots,N_d) \text{ for all $B$ sufficiently large}
\end{equation}
almost surely. By the preceding calculations, this suffices to finish the proof.

We consider the auxiliary numbers $M_i \df \lfloor (1 - \epsilon/2) B_i\rfloor$ ($i\in D$). We also let
\begin{align*}
\epsilon_0 &= \min_{i\in D} \min_{x\in \{0,1\}} \min_{\substack{a\in A_i \\ x\notin \phi_{i,a}([0,1])}} \dist\big(x,\phi_{i,a}([0,1])\big),&
C &= -\log(\epsilon_0).
\end{align*}
If $B$ is sufficiently large (depending on $\epsilon$), then
\begin{equation}
\label{Xibound}
X_i(\omega\given {M_i}) < \chi_i(\bfR_{B_i}) - C = B - C = -\log(\rho/\epsilon_0) \all i\in D.
\end{equation}
Now fix $i\in D$, and consider the sequence of random events
\[
\big(E_n(i) \df \big[\phi_{\omega_n,i}\circ\phi_{\omega_{n + 1},i}([0,1]) \subset (0,1)\big]\big)_{n\in\N}.
\]
These events are not independent, but the subsequences corresponding to even and odd indices are both sequences of independent events. So again by \cite[Corollary 1.8 in the Appendix]{BenoistQuint_book}, we have
\[
\#\{n \leq N : E_n(i) \text{ holds}\} = \sum_{n = 1}^N p_n + o(N),
\]
almost surely, where $p_n$ is the probability of $E_n$. In particular, for all $j\in D$
\[
\#\{N_j < n < M_j : E_n(i) \text{ holds}\} = \sum_{n = N_j}^{M_j} p_n + o(M_j).
\]
Letting
\begin{equation}
\label{fdef}
f(\pp) = \pp\times\pp(\{(\aa,\bb)\in E^2 : \phi_{\aa,i}\circ\phi_{\bb,i}([0,1]) \subset (0,1)\}),
\end{equation}
we have $p_n = f(\rr_n) + o(1)$ and thus
\[
\#\{N_j < n < M_j : E_n(i) \text{ holds}\} = \int_{(1 - \epsilon)B_j}^{(1 - \epsilon/2)B_j} f(\rr_b) \;\dee b + o(B).
\]
Now without loss of generality suppose that $\phi_{\aa,i} \circ \phi_{\bb,i}([0,1]) \subset (0,1)$ for some $\aa,\bb\in E$. (If not, then there exists $x \in \{0,1\}$ such that $\phi_{\aa,i}(x) = x$ for all $\aa\in E$, in which case the coordinate $i$ can be ignored since its value is constant over the entire sponge $\Lambda_\Phi$.) Then $f(\pp) > 0$ for all $\pp\in \PP$ such that $\pp(\aa) > 0$ for all $\aa\in E = E_\rr$. So since $\rr$ is nondegenerate, we have
\[
\int_{(1 - \epsilon)B_i}^{(1 - \epsilon/2)B_i} f(\rr_b) \;\dee b \geq \delta B
\]
for some $\delta > 0$ depending on $\epsilon$. So we have
\begin{equation}
\label{Enholds}
\{N_j < n < M_j : E_n(i) \text{ holds}\} \neq \emptyset
\end{equation}
for all $B$ sufficiently large (depending on $\epsilon$).

Now fix $\tau \in E^\N$ such that $\pi(\tau) \in B(\pi(\omega),\rho)$, and we will show that $\tau \in B_\omega(N_1,\ldots,N_d)$. Indeed, by contradiction, suppose that $\tau \notin [\omega\given {N_j}]_{\{j\}}$ for some $j\in D$, and let
\[
I = I(\what\bfR_{B_j},B/B_j) = \{i\in D : B_i \geq B_j\}
\]
(cf. \eqref{Ipx}). Since $\rr$ is good, so is $I$. Moreover, since $j\in I$, we have $\tau \notin [\omega\given {N_j}]_I$. Write $N = N_j$ and $M = M_j$. Then
\begin{align*}
\rho \geq \dist(\pi_I(\omega),\pi_I(\tau))
&\geq \dist\big(\phi_{\omega\given M,I}([0,1]^I),\R^I\butnot \phi_{\omega\given N,I}((0,1)^I)\big) \since{$I$ is good}\\
&\geq \epsilon_0 \min_{i\in I} \big|\phi_{\omega\given M,i}'\big| \by{\eqref{Enholds}}\\
&= \epsilon_0 \exp\big(-\max_{i\in I} X_i(\omega\given {M_j})\big)\\
&\geq \epsilon_0 \exp\big(-\max_{i\in I} X_i(\omega\given {M_i})\big), \since{$B_i \geq B_j \all i\in I$}
\end{align*}
which contradicts \eqref{Xibound}. This demonstrates \eqref{ETSQB}, completing the proof. 
\end{proof}

\section{Hausdorff and dynamical dimensions of self-affine sponges}
\label{sectionHDDD}

In this section we compute the Hausdorff and dynamical dimensions of a self-affine sponge by proving Theorem \ref{theoremgood2}, which implies Theorems \ref{theoremDDbernoulli} and \ref{theoremHDcompute}.

\begin{proof}[Proof of Theorem \ref{theoremgood2}]
Let $\rr\in\RR$ be a good cycle. Fix $0 < \epsilon < 1$, and let
\[
\ss_b = (1 - \epsilon)\rr_{b^{1 - \epsilon}} + \epsilon \what\bfR_{b^{1 - \epsilon}},
\]
so that $\bfS_B = B^\epsilon \bfR_{B^{1 - \epsilon}}$. Since $\rr$ is a good cycle, so is $\ss$. For all $\aa\in E_\ss = E_\rr$ and $b > 0$, we have $\what\bfR_{b^{1 - \epsilon}}(\aa) > 0$ and thus $\ss_b(\aa) > 0$, so $\ss$ is nondegenerate. Thus by Theorem \ref{theoremgood1}, we have
\[
\HD(\Phi) \geq \HD(\nu_\ss) = \delta(\ss) \tendsto{\epsilon \to 0} \delta(\rr).
\]
Taking the supremum over all good $\rr\in\RR$ proves the left-hand inequality of \eqref{HDbounds}. On the other hand, the left-hand inequality of \eqref{DDbounds} is immediate from Theorem \ref{theoremgood1}.

We will now prove the right-hand inequalities of \eqref{HDbounds} and \eqref{DDbounds}. For each $\rr\in\RR$ and $\epsilon > 0$, we let
\[
S_{\rr,\epsilon} = \big\{\xx \in \Lambda_\Phi : \underline\pdim(\xx,\nu_\rr) \leq \delta(\rr) + \epsilon\big\},
\]
where $\delta(\rr)$ denotes the right-hand side of \eqref{deltar}. By Theorem \ref{theoremRT}, we have $\HD(S_{\rr,\epsilon}) \leq \delta(\rr) + \epsilon$. Now for each rational $\lambda \geq 1$ let $\QQ_\lambda$ be a countable dense subset of $\RR_\lambda$, and let $\QQ = \bigcup_{1 \leq \lambda \in \Q} \QQ_\lambda$. Then since Hausdorff dimension is $\sigma$-stable, the sets
\begin{align*}
S_1 &\df \bigcap_{\epsilon > 0} \bigcup_{\rr\in\QQ} S_{\rr,\epsilon}\\
S_2 &\df \bigcap_{\epsilon > 0} \bigcup_{\pp\in\QQ_1} S_{\pp,\epsilon}
\end{align*}
satisfy
\begin{align*}
\HD(S_1) &\leq \sup_{\rr\in\QQ} \delta(\rr),\\
\HD(S_2) &\leq \sup_{\pp\in\QQ_1} \delta(\pp).
\end{align*}
To complete the proof, we need to show that
\begin{align} \label{S1}
\HD(\Phi) &\leq \HD(S_1),\\ \label{S2}
\DynD(\Phi) &\leq \HD(S_2).
\end{align}
We will prove \eqref{S1} first, since afterwards it will be easy to modify the proof to show \eqref{S2}. Fix $\omega\in E^\N$, and we will show that $\pi(\omega) \in S_1$. For each $N\in\N$ let
\begin{align}
\label{PNdef}
\bfP_N &= \sum_{n = 1}^N \delta_{\omega_n},&
\what\bfP_N &= \frac{1}{N} \bfP_N.
\end{align}
If $\pp$ is a signed measure on $E$, then we let
\[
\|\pp\| = \sum_{\aa\in E} |\pp(\aa)|.
\]

\begin{claim}
\label{claimexistsr}
For all $C > 1$ and $\epsilon > 0$, there exist $1 < \lambda \in \Q$ and $\rr \in \QQ_\lambda$ such that for all $B\in [1,\lambda]$,
\begin{equation}
\label{existsr}
\liminf_{k\to\infty}\sup_{M\in [C^{-1} \lambda^k B,C \lambda^k B]} \|\what\bfR_M - \what\bfP_M\| \leq \epsilon.
\end{equation}
Moreover, $\rr$ may be taken so that $\rr_b \in \PP^*$ for all $b > 0$, where
\[
\PP^* \df \{\pp\in\PP : \pp(\aa) > 0 \all \aa \in E\}.
\]
\end{claim}
\begin{subproof}
By compactness, there is a sequence of $N$s such that for all $B\in\Q^+$ we have
\begin{equation}
\label{subseq}
\frac{1}{N}\bfP_{N B} \dashrightarrow \bfQ_B,
\end{equation}
where $\dashrightarrow$ indicates convergence along this sequence. Since the map $\Q^+\ni B\mapsto \bfQ_B$ is increasing and uniformly continuous (in fact $1$-Lipschitz), it can be extended to an increasing continuous map $\R^+ \ni B\mapsto \bfQ_B$. Note that $\bfQ_B(E) = B$ for all $B\in\R^+$. Write $\what\bfQ_B = B^{-1}\bfQ_B \in \PP$.

Fix $0 < \epsilon_3 < \epsilon_2 < 1$ small to be determined. For each $t\in\R$ write $\qq(t) = \what\bfQ_{\exp(t)}$. For all $t_2 > t_1$, we have
\begin{align*}
\|\qq(t_2) - \qq(t_1)\|
&= \|e^{-t_2} \bfQ_{\exp(t_2)} - e^{-t_1} \bfQ_{\exp(t_1)}\|\\
&= \|e^{-t_2} (e^{t_1} \aa + (e^{t_2} - e^{t_1}) \bb) - e^{-t_1} (e^{t_1} \aa)\| \note{for some $\aa,\bb\in\PP$}\\
&= \|e^{-t_2} (e^{t_2} - e^{t_1}) (\bb - \aa)\|\\
&\leq 2 e^{-t_2} (e^{t_2} - e^{t_1}) \leq 2 (t_2 - t_1), \since{$\|\aa\| = \|\bb\| = 1$}
\end{align*}
i.e. $\qq$ is $2$-Lipschitz. By the Arzela--Ascoli theorem the collection of all $2$-Lipschitz maps from $\R$ to $\PP$ is compact in the topology of locally uniform convergence. Since the translated paths $t\mapsto \qq(T + t)$ ($T\in\R$) are members of this collection, it follows that there exist $T_1,T_2\in\R$ with $\rho_1 \df T_2 - T_1 \geq \log(C)$, such that for all $t\in [-\log(C),\log(C)]$, $\|\qq(T_2 + t) - \qq(T_1 + t)\| \leq \epsilon_3$. Let $A_1 = \exp(T_1)$, $A_2 = \exp(T_2)$, and $\lambda_1 = A_2/A_1 = \exp(\rho_1) \geq C$. Then
\begin{equation}
\label{A1A2}
\text{for all $B\in [C^{-1},C]$, we have $\|\what\bfQ_{A_2 B} - \what\bfQ_{A_1 B}\| \leq \epsilon_3$.}
\end{equation}
Now for each $B\in [A_1,A_2]$, let $\bfS_B = (1 - \epsilon_2)\bfQ_B + \epsilon_2 B\uu$ and $\what\bfS_B = B^{-1} \bfS_B = (1 - \epsilon_2)\what\bfQ_B + \epsilon_2 \uu$, where $\uu\in \PP$ is the normalized uniform measure on $E$. (We will later define $\bfS_B$ for $B\notin [A_1,A_2]$ as well, but not with this formula.) Let $\delta = \#(E) \epsilon_3/\epsilon_2 > 0$. Then
\begin{align*}
(1 + \delta)\what\bfS_{A_1} - \what\bfS_{A_2}
&= (1 - \epsilon_2)((1 + \delta)\what\bfQ_{A_1} - \what\bfQ_{A_2}) + \delta\epsilon_2 \uu\\
&\geq (1 - \epsilon_2)(\what\bfQ_{A_1} - \what\bfQ_{A_2}) + \delta\epsilon_2 \uu\\
&\geq -(1 - \epsilon_2)\|\what\bfQ_{A_2} - \what\bfQ_{A_1}\|\#(E)\uu + \delta\epsilon_2\uu\\
&\geq -\#(E)\epsilon_3\uu + \delta\epsilon_2\uu = \0. \by{\eqref{A1A2}}
\end{align*}
Let $\lambda \in [(1 + \delta) \lambda_1,(1 + 2 \delta) \lambda_1]$ be a rational number, so that $\lambda \bfS_{A_1} \geq \bfS_{A_2}$. We let $\bfS_{\lambda A_1} = \lambda \bfS_{A_1}$, and we define $B\mapsto \bfS_B$ on the interval $[A_2,\lambda A_1]$ by linear interpolation:
\[
\bfS_B = \bfS_{A_2} + \frac{B - A_2}{\lambda A_1 - A_2}(\lambda \bfS_{A_1} - \bfS_{A_2}) \text{ for all }B\in [A_2,\lambda A_1],
\]
and as before we let $\what\bfS_B = B^{-1} \bfS_B$. Then $\what\bfS_{\lambda A_1} = \what\bfS_{A_1}$, so there is a unique exponentially $\lambda$-periodic extension $\what\bfS:(0,\infty)\to\PP$. We let $\bfS_B = B \what\bfS_B$, and note that $\bfS$ is increasing.

Fix $B\in [\lambda A_1, C\lambda A_1]$. Since $\lambda_1 \geq C$, we have
\begin{eqnarray*}
\what\bfS_B = \what\bfS_{B/\lambda}
&\underset{\epsilon_2\to 0}{\sim_\plus}& \what\bfQ_{B/\lambda}\hspace{1 in} \sincevar{$B/\lambda \in [A_1, C A_1] \subset [A_1,A_2]$}\\
&\underset{\epsilon_3\to 0}{\sim_\plus}& \what\bfQ_{(\lambda_1/\lambda)B} \byvar{\eqref{A1A2}}\\
&\underset{\delta\to 0}{\sim_\plus}& \what\bfQ_B, \sincevar{$1 \leq \lambda/\lambda_1 \leq 1 + 2\delta$}
\end{eqnarray*}
where $X \sim_\plus Y$ means that the distance between $X$ and $Y$ tends to zero as the appropriate limit is taken. Similar logic applies if $B\in [C^{-1} A_1,A_1]$, and the cases $B\in [A_1,A_2]$ and $B\in [A_2,\lambda A_1]$ are even easier. So
\begin{equation}
\label{epsilon2delta}
\sup_{B\in [C^{-1} A_1,C\lambda A_1]} \|\what\bfS_B - \what\bfQ_B\| \tendsto{\epsilon_2,\delta \to 0} 0.
\end{equation}
For each $N$, let $k = k_N \in \N$ be chosen so that $\lambda^{-k} N \in [1,\lambda]$. After extracting a subsequence from the sequence along which \eqref{subseq} converges, we can assume that
\begin{equation}
\label{xdef}
\lambda^{-k} N \dashrightarrow x \in [1,\lambda].
\end{equation}
Now let $\psi:\R\to \Rplus$ be a smooth approximation of the Dirac delta function, let
\[
\what\bfT_{x B} = \int \what\bfS_{e^t B} \psi(t)\;\dee t,
\]
and let $\tt_b = (\del/\del b)[b\what\bfT_b]$. Then $\tt\in\RR_\lambda$, and by choosing $\psi$ appropriately we can guarantee
\begin{equation}
\label{TxB}
\sup_{B > 0} \|\what\bfT_{x B} - \what\bfS_B\| < \epsilon_2.
\end{equation}
Finally, let $\rr\in\QQ_\lambda$ be an approximation of $\tt$, such that $\rr_b \in \PP^*$ for all $b > 0$, and
\begin{equation}
\label{rbtb}
\sup_{b > 0} \|\rr_b - \tt_b\| < \epsilon_2.
\end{equation}
Now fix $B\in [1,\lambda]$, let $N$ be large, and let $k = k_N$. Let $k' \in \Z$ be chosen so that $N A_1 \leq \lambda^{k'} B \leq N \lambda A_1$. Now fix $M \in [C^{-1} \lambda^{k'} B, C \lambda^{k'} B]$, and let $B' = M/N$. By our choice of $k'$, we have $C^{-1} A_1 \leq B' \leq C\lambda A_1$. Thus
\begin{eqnarray*}
\what\bfP_M
= \what\bfP_{N B'}
&\underset{N\dashrightarrow\infty}{\sim_\plus}& \what\bfQ_{B'}
\byvar{\eqref{subseq}}\\
&\underset{\epsilon_2,\delta\to 0}{\sim_\plus}& \what\bfS_{B'}
\byvar{\eqref{epsilon2delta}}\\
&\underset{\epsilon_2 \to 0}{\sim_\plus}& \what\bfR_{x B'}
\byvar{\eqref{TxB} and \eqref{rbtb}}\\
&\underset{N\dashrightarrow\infty}{\sim_\plus}& \what\bfR_{\lambda^{-k} N B'}
= \what\bfR_M,
\byvar{\eqref{xdef}}\\
\end{eqnarray*}
which completes the proof of the claim.
\end{subproof}

Now fix $\lambda > 1$, $B\in [1,\lambda]$, and $k\in\N$. For each $i\in D$, let $N_i = \lfloor \lambda^k B_i\rfloor$, where $B_i$ is given by \eqref{Bidef}. Since $\chi_i$ is bounded from above and below on $\PP$, there exists a constant $C \geq 1$ (independent of $\lambda$, $B$, and $k$) such that $N_i \in [C^{-1} \lambda^k B,C \lambda^k B]$. Fix $\epsilon > 0$ and let $1 < \lambda \in \Q$ and $\rr\in \QQ_\lambda$ be as in Claim \ref{claimexistsr}. Then
\begin{align*}
X_i(\omega\given {N_i})
\; &=_\pt -\sum_{n = 1}^{N_i} \log|\phi_{\omega_n,i}'|
= \chi_i(\bfP_{N_i})\\
&\sim_\times \chi_i(\bfR_{N_i}) \note{as $\epsilon \to 0$}\\
&\sim_\times \chi_i(\bfR_{\lambda^k B_i}) \note{as $k\to\infty$}\\
&=_\pt \lambda^k \chi_i(\bfR_{B_i})
= \lambda^k B, \by{\eqref{Bidef}}
\end{align*}
where $X \sim_\times Y$ means that $X/Y \to 1$ as the appropriate limit is taken. So for some $\delta_2 > 0$ such that $\delta_2 \to 0$ as $\epsilon \to 0$ and $k \to \infty$, we have
\[
X_i(\omega\given {N_i}) \geq (1 - \delta_2)\lambda^k B.
\]
Letting $\rho_k = \exp(-(1 - \delta_2) \lambda^k B)$, we have $B_\omega(N_1,\ldots,N_d) \subset \pi^{-1}(B(\pi(\omega),\rho_k))$ (cf. \eqref{Bomegadef}) and thus
\begin{align}
\label{Xidef}
-\log\nu_\rr\big(B(\pi(\omega),\rho_k)\big)
&\leq -\log\mu_\rr\big(B_\omega(N_1,\ldots,N_d)\big)
= -\sum_{n\in\N} \log\rr_n([\omega_n]_{\{i\in D : n\leq N_i\}}).
\end{align}
In order to estimate the right-hand side, let $\ss:(0,\infty)\to \PP$ be a piecewise constant and exponentially periodic approximation of $\rr$. Let $F$ denote the range of $\ss$, and note that $F$ is finite. Then since $\rr_b \in \PP^*$ for all $b > 0$, we can continue the calculation as follows:
\begin{align*}
&\sim_\times -\sum_{n\in\N} \log\ss_n([\omega_n]_{\{i\in D : n\leq N_i\}}) \note{as $\ss\to\rr$}\\
&=_\pt -\sum_{\smallemptyset \neq I \subset D} \sum_{\tt \in F} \sum_{\substack{n\in\N \\ \ss_n = \tt \\ \{i\in D : n\leq N_i\} = I}} \log\tt([\omega_n]_I).
\end{align*}
Now for each $\emptyset \neq I \subset D$ and $\tt \in F$, the set
\[
\big\{n \geq \epsilon \lambda^k B : \ss_n = \tt, \{i \in D : n \leq N_i\} = I \big\}
\]
can be written as the union of at most $C_2$ disjoint intervals, where $C_2$ depends only on $\epsilon$ and $\ss$. Write this collection of intervals as $\II(I,\tt)$.

We continue the calculation begun in \eqref{Xidef}, using the notation $k\dashrightarrow \infty$ to denote convergence along the sequence tending to the liminf in \eqref{existsr}:
\begin{align*}
&\sim_\times -\sum_{\smallemptyset \neq I \subset D} \sum_{\tt \in F} \sum_{\substack{n\geq \epsilon \lambda^k B \\ \ss_n = \tt \\ \{i\in D : n\leq N_i\} = I}} \log\tt([\omega_n]_I) \note{as $\epsilon\to 0$}\\
&=_\pt -\sum_{\smallemptyset \neq I \subset D} \sum_{\tt\in F} \sum_{\OC{M_1}{M_2} \in \II(I,\tt)} \int \log\tt([\aa]_I) \;\dee[\bfP_{M_2} - \bfP_{M_1}](\aa)
\end{align*}
\begin{align*}
&\sim_\times -\sum_{\smallemptyset \neq I \subset D} \sum_{\tt\in F} \sum_{\OC{M_1}{M_2} \in \II(I,\tt)} \int \log\tt([\aa]_I) \;\dee[\bfR_{M_2} - \bfR_{M_1}](\aa) \note{as $\epsilon \to 0$ and $k\dashrightarrow \infty$}\\
&=_\pt -\sum_{n \geq \epsilon \lambda^k B} \int_n^{n + 1} \int \log\ss_n([\aa]_{\{i\in D : n \leq N_i\}}) \;\dee\rr_b(\aa) \;\dee b\\
&\sim_\times -\sum_{n\in\N} \int_n^{n + 1} \int \log\rr_n([\aa]_{\{i\in D : n \leq N_i\}}) \;\dee\rr_b(\aa) \;\dee b \note{as $\epsilon \to 0$ and $\ss\to\rr$}\\
&\sim_\times -\iint \log\rr_b([\aa]_{\{i\in D : b \leq \lambda^k B_i\}}) \;\dee\rr_b(\aa) \;\dee b \note{as $k\to\infty$}\\
&=_\pt \int h(\{i \in D : b \leq \lambda^k B_i\};\rr_b) \;\dee b
= \lambda^k B \delta(\rr,B).
\end{align*}
Dividing by the asymptotic $\lambda^k B \sim_\times -\log(\rho_k)$ (valid as $\epsilon\to 0$) and letting $k\dashrightarrow\infty$ and $\ss\to\rr$ shows that
\[
\underline\pdim(\pi(\omega),\nu_\rr) \leq \liminf_{k\to\infty} \frac{\log\nu_\rr\big(B(\pi(\omega),\rho_k)\big)}{\log(\rho_k)} \leq (1 + o(1)) \delta(\rr,B),
\]
where the $o(1)$ term decays to zero as $\epsilon \to 0$. Taking the infimum over $B\in [1,\lambda]$ gives
\[
\underline\pdim(\pi(\omega),\nu_\rr) \leq (1 + o(1)) \delta(\rr),
\]
which proves that $\pi(\omega) \in S_1$, demonstrating \eqref{S1}.

Now we prove \eqref{S2}. Let $\Omega$ be the set of all $\omega\in E^\N$ such that the limit $\lim_{N\to\infty} \what\bfP_N$ exists, where $\what\bfP_N \in \PP$ is given by \eqref{PNdef}. By the ergodic theorem, every invariant measure gives full measure to $\Omega$, so $\DynD(\Phi) \leq \HD(\pi(\Omega))$. Now for each $\omega\in \Omega$, we can choose $\rr = \pp \in \QQ_1\cap \PP^*$ satisfying \eqref{existsr}, namely any approximation to the limit $\lim_{N\to\infty} \what\bfP_N$. The remainder of the argument (i.e. everything after the proof of Claim \ref{claimexistsr}) is still applicable, and shows that $\underline\pdim(\pi(\omega),\nu_\pp) \leq (1 + o(1)) \delta(\pp)$, so $\pi(\omega) \in S_2$. Since $\omega$ was arbitrary, we have $\pi(\Omega) \subset S_2$, demonstrating \eqref{S2}.
\end{proof}

\section{Continuity of dimension functions}
\label{sectioncontinuous}

In this section we prove the continuity of the Hausdorff and dynamical dimensions as functions of the defining IFS, i.e. Theorem \ref{theoremcontinuous}.

\begin{theorem}[Generalization of Theorem \ref{theoremcontinuous}]
The functions
\begin{align*}
\Phi &\mapsto \sup_{\rr\in\RR} \delta(\rr),&
\Phi &\mapsto \sup_{\pp\in\PP} \delta(\pp)
\end{align*}
are continuous on the space of all diagonal IFSes.
\end{theorem}
\begin{proof}
It is easy to see that the maps
\begin{align*}
(\Phi,i,\pp) &\mapsto \chi_i(\pp),&
(\Phi,I,\pp) &\mapsto h_I(\pp)
\end{align*}
are continuous. Applying \eqref{deltap1} shows that the map
\[
(\Phi,\pp) \mapsto \delta(\pp)
\]
is continuous. Since $\PP$ is compact, it follows that the map $\Phi \mapsto \sup_{\pp\in\PP} \delta(\pp)$ is continuous.

Now if we endow $\RR$ with the topology of locally uniform convergence, then the maps
\begin{align*}
(\Phi,i,\rr,B) &\mapsto B_i,&
(\Phi,\rr,B) &\mapsto \delta(\rr,B)
\end{align*}
are continuous. Since the infimum in \eqref{deltar} is taken over a compact set, it follows that the map
\[
(\Phi,\lambda,\rr) \mapsto \delta(\rr)
\]
is continuous. Here we need to include $\lambda$ as an input because of its appearance in the formula \eqref{deltar}.

Now we define the \emph{exponential Lipschitz constant} of a cycle $\rr\in\RR$ to be the Lipschitz constant of the periodic function $t\mapsto \rr_{\exp(t)}$. Note that although some elements of $\RR$ have infinite exponential Lipschitz constant, we can choose the countable dense subsets $\QQ_\lambda \subset \RR_\lambda$ appearing in the proof of Theorem \ref{theoremgood2} so that all elements of $\QQ \df \bigcup_{1 \leq \lambda \in \Q} \QQ_\lambda$ have finite exponential Lipschitz constant. For each $k > 1$, let $\RR_{\lambda,k}$ (resp. $\QQ_{\lambda,k}$) denote the set of all cycles $\rr\in\RR_\lambda$ (resp. $\rr\in\QQ_\lambda$) with exponential Lipschitz constant $\leq k$. Then by the Arzela--Ascoli theorem, the set
\[
\coprod_{\lambda\in [1,k]} \RR_{\lambda,k} = \{(\lambda,\rr) : \lambda \in [1,k], \; \rr \in \RR_{\lambda,k}\}
\]
is compact, and thus for each $k$ the map
\[
\Phi \mapsto \delta_k \df \sup_{\rr \in \bigcup_{\lambda \in [1,k]} \RR_{\lambda,k}} \delta(\rr)
\]
is continuous. To complete the proof, we need to show that the convergence
\[
\delta_k \tendsto{k\to\infty} \sup_{\rr\in\RR} \delta(\rr)
\]
is locally uniform with respect to $\Phi$.

Indeed, fix $\epsilon > 0$, and let $0 < \epsilon_3 < \epsilon_2 < 1$ be as in the proof of Claim \ref{claimexistsr}. Then:
\begin{itemize}
\item The numbers $T_1,T_2 \in \R$ appearing in the proof of Claim \ref{claimexistsr} may be chosen so that $\rho_1 \df T_2 - T_1$ is bounded depending only on $C$, $\epsilon_3$, and $\#(E)$. Since $\lambda$ can be bounded in terms of $\rho_1$, this shows that the $\lambda$ appearing in the conclusion of Claim \ref{claimexistsr} can be bounded in terms of the $C$ and $\epsilon$ that appear in the hypotheses. Now $C$ depends only on the maximum and minimum of the function $D\times \PP \ni (i,\pp) \mapsto \chi_i(\pp)$, so it is bounded when $\Phi$ ranges over a compact set. So $\lambda$ can be bounded in terms of $\epsilon$, assuming that $\Phi$ ranges over a compact set.
\item The exponential Lipschitz constant of the function $\tt$ appearing in the proof of Claim \ref{claimexistsr} can be bounded in terms of the $C^2$ norm of the smooth function $\psi$. The function $\psi$ depends only on $\epsilon_2$, which in turn depends only on $\epsilon$. Moreover, an approximation $\rr\in\QQ_\lambda$ of $\tt$ satisfying \eqref{rbtb} can be found with exponential Lipschitz constant bounded in terms of the Lipschitz norm of $\tt$. So the exponential Lipschitz constant of $\rr$ is bounded in terms of $\epsilon$.
\item The rate of convergence of the $o(1)$ term to $0$ at the end of the proof of Theorem \ref{theoremgood2} is locally uniform with respect to $\Phi$ as $\epsilon \to 0$.
\end{itemize}
Thus the proof of Theorem \ref{theoremgood2} actually shows that
\[
\HD(\Phi) \leq \HD\left(\bigcap_{\epsilon > 0} \bigcup_{\lambda \in \Q\cap [1,k(\epsilon)]} \bigcup_{\rr\in \QQ_{\lambda,k(\epsilon)}} S_{\rr,\epsilon}\right) \leq \inf_{\epsilon > 0} [\delta_{k(\epsilon)} + \epsilon]
\]
for some function $k$ that can be taken to be independent of $\Phi$ as $\Phi$ ranges over a compact set. Thus if $\Lambda_\Phi$ is good, then
\begin{equation}
\label{deltakepsilon}
\delta_{k(\epsilon)} \geq \sup_{\rr\in\RR} \delta(\rr) - \epsilon,
\end{equation}
which completes the proof in this case. If $\Lambda_\Phi$ or its perturbations are not good, then we may justify the inequality \eqref{deltakepsilon} by appealing to the existence of a good sponge $\Lambda_\Psi$ with good perturbations, indexed by the same set $E$, such that $|\psi_{i,a}'| = |\phi_{i,a}'|^\alpha$ for all $i\in D$ and $a\in A_i$. Here $\alpha > 0$ must be chosen large enough so that $\sum_{a\in A_i} |\phi_{i,a}'|^\alpha < 1$ for all $i\in D$, which guarantees the existence of a base IFS $\Psi_i$ whose perturbations satisfy the open set condition. It is readily verified that $\delta_\Psi(\rr) = \delta_\Phi(\rr)/\alpha$ for all $\rr\in\RR$, so that \eqref{deltakepsilon} holds for $\Phi$ if and only if it holds for $\Psi$.
\end{proof}

\section{Special cases where $\HD(\Phi) = \DynD(\Phi)$}
\label{sectionequality}

In this section we give new proofs of Theorems \ref{theoremLGB} and \ref{theoremBMhigher}, i.e. equality of the Hausdorff and dynamical dimensions in certain special cases, based on the results of the previous sections. Both of the theorems can now be stated in somewhat greater generality than they were in the introduction.

\begin{theorem}[Generalization of Theorem \ref{theoremBMhigher}]
\label{theoremBMhigher2}
Let $\Lambda_\Phi$ be a good sponge such that for all $i\in D$, the map $A_i \ni a \mapsto |\phi_{i,a}'|$ is constant. Then $\HD(\Phi) = \DynD(\Phi)$.
\end{theorem}
\begin{proof}
Fix $\rr\in\RR$, and we will show that $\delta(\rr) \leq \DynD(\Phi)$. For each $i\in D$, let $r_i > 0$ be the constant such that $|\phi_{i,a}'| = r_i$ for all $a\in A_i$, and let $X_i = -\log(r_i)$. For all $B > 0$ and $i\in D$, we have
\[
B = \chi_i(\bfR_{B_i}) = X_i B_i,
\]
i.e. $B_i = B/X_i$. Now without loss of generality suppose that $X_1 \leq \cdots \leq X_d$. Then
\begin{align*}
\delta(\rr) &\leq \frac{1}{\log(\lambda)} \int_1^\lambda \delta(\rr,B) \;\frac{\dee B}{B} \by{\eqref{deltar}}\\
&\leq \frac{1}{\log(\lambda)} \int_1^\lambda \sum_{i\in D} \frac{h(I_{\leq i}\given I_{\leq i - 1};\what\bfR_{B_i})}{\chi_i(\what\bfR_{B_i})} \;\frac{\dee B}{B} \by{\eqref{deltarB2}}\\
&= \frac{1}{\log(\lambda)} \int_1^\lambda \sum_{i\in D} \frac{h(I_{\leq i}\given I_{\leq i - 1};\what\bfR_A)}{\chi_i(\what\bfR_A)} \;\frac{\dee A}{A} \note{letting $A = B_i$}\\
&= \frac{1}{\log(\lambda)} \int_1^\lambda \delta(\what\bfR_A) \;\frac{\dee A}{A} \by{\eqref{deltap2}}\\
&\leq \frac{1}{\log(\lambda)} \int_1^\lambda \DynD(\Phi) \;\frac{\dee A}{A} = \DynD(\Phi). \by{\eqref{DDbernoulli} and \eqref{deltap1}}
\end{align*}
The key step in this proof is the substitution $A = B_i = B/X_i$, which is valid because $\dee B_i/B_i = \dee B/B$. In general, when $B_i$ and $B$ are only related by the formula \eqref{Bidef}, the relation $\dee B_i/B_i = \dee B/B$ is not valid, and that is the reason that this proof does not work in the general case.
\end{proof}

\begin{theorem}[Generalization of Theorem \ref{theoremLGB}]
\label{theoremDGbound}
For every good sponge $\Lambda_\Phi \subset [0,1]^d$, we have
\[
\HD(\Phi) \leq \max(1,d - 1) \DynD(\Phi).
\]
In particular, if $d \leq 2$ then $\HD(\Phi) = \DynD(\Phi)$.
\end{theorem}
\begin{proof}
Fix $\rr\in\RR$, and we will show that $\delta(\rr) \leq \max(1,d - 1)\DynD(\Phi)$. For each $B > 0$ and $i\in D$ we let
\[
J_{i,B} = \{j\in D : \chi_j(\what\bfR_B) \leq^* \chi_i(\what\bfR_B)\},
\]
where the star on the inequality means that in the case of a tie, we determine whether or not the inequality is true using an arbitrary but fixed ``tiebreaker'' total order $\prec$ on $D$: we declare the inequality to be true if $j \precneq i$, and false if $j \succ i$. Then we let
\[
f_i(B) = \frac{h(J_{i,B}\cup\{i\} \given J_{i,B};\what\bfR_B)}{\chi_i(\what\bfR_B)}\cdot
\]
Now,
\begin{itemize}
\item If $\chi_1(\what\bfR_B) \leq^* \cdots \leq^* \chi_d(\what\bfR_B)$, then $J_{i,B} = I_{\leq i - 1}$, and so by Theorem \ref{theoremDDbernoulli} and Proposition \ref{propositionXinc},
\begin{equation}
\label{fiB}
\DynD(\Phi) \geq \delta(\what\bfR_B) = \sum_{i\in D} f_i(B).
\end{equation}
\item If $B_1 \geq^* \cdots \geq^* B_d$, then $\chi_j(\bfR_{B_i}) \leq^* \chi_i(\bfR_{B_i}) \leq^* \chi_{j'}(\bfR_{B_i})$ for all $j < i < j'$, so $J_{i,B_i} = I_{\leq i - 1}$, and thus by  Proposition \ref{propositionXinc}, we have
\begin{equation}
\label{fiBi}
\delta(\rr) \leq \delta(\rr,B) \leq \sum_{i\in D} f_i(B_i),
\end{equation}
where $B_1,\ldots,B_d > 0$ are as in \eqref{Bidef}.
\end{itemize}
Both of these hypotheses can be attained by appropriately permuting $D$, assuming that the tiebreaker total order is getting permuted as well. So since the formulas \eqref{fiB} and \eqref{fiBi} are invariant under permutations of $D$, they are true regardless of how the numbers $\chi_i(\what\bfR_B)$ ($i\in D$) and $B_i$ ($i\in D$) are ordered.

Now fix $\epsilon > 0$, and let $B > 0$ be chosen so that $f_1(B_1) \leq \inf(f_1) + \epsilon$. (This is possible because the map $B \mapsto B_1$ is a homeomorphism of $(0,\infty)$.) If $d \geq 2$, then we get
\[
f_1(B_1) + f_2(B_2) \leq f_1(B_2) + \epsilon + f_2(B_2) \leq \DynD(\Phi) + \epsilon
\]
and thus
\[
\delta(\rr) \leq \sum_{i\in D} f_i(B_i) \leq \DynD(\Phi) + \epsilon + \sum_{i = 3}^d f_i(B_i) \leq (d - 1) \DynD(\Phi) + \epsilon.
\]
Since $\rr$ and $\epsilon$ were arbitrary, we get $\HD(\Phi) \leq (d - 1)\DynD(\Phi)$. If $d = 1$, then $\HD(\Phi) = \DynD(\Phi)$, so in any case $\HD(\Phi) \leq \max(1,d - 1)\DynD(\Phi)$.
\end{proof}

\begin{remark}
\label{remark2D3D}
This new way of proving Theorem \ref{theoremLGB} sheds light on the question of why there is a difference between the two-dimensional and three-dimensional settings. Namely, since we used the assumption $d = 2$ only at the last possible moment, the proof clarifies exactly how the assumption is needed in the argument.

At a very abstract level, the difference between the two-dimensional and three-dimen\-sion\-al case can be described as follows: The Hausdorff dimension of a ``homogeneous'' non-invariant measure (such as a pseudo-Bernoulli measure) is equal to the lim inf of its dimension at different length scales. At each length scale, the dimension is equal to the sum of the coordinatewise dimensions at that scale. So if $\delta_i$ is the coordinatewise dimension as a function of the length scale $\rho$, then
\[
\HD(\text{non-invariant measure}) = \liminf_{\rho\to 0} \sum_i \delta_i(\rho).
\]
Now, the existence of this non-invariant homogeneous measure will allow us to deduce the existence of certain invariant measures, namely there exist continuously varying tuples of length scales $(\rho_1,...,\rho_d)\to 0$ such that there is some invariant measure which for all $i$ has the same behavior as the non-invariant measure in coordinate $i$ and length scale $\rho_i$. The dimension of such a measure would be
\[
\HD(\text{invariant measure}) = \sum_i \delta_i(\rho_i).
\]
Obviously, the problem with comparing these two formulas is that the $\rho_i$s may be different from each other. In dimension $1$, there is only one number $\rho_i$ so there is no issue. But we can handle one more dimension using the fact that the first formula has a lim inf instead of a lim sup. Namely, we can choose a value of $\rho$ so as to minimize one of the numbers $\delta_i(\rho)$, for concreteness say $\delta_1(\rho)$. This handles the first coordinate, and we can handle the second coordinate by choosing the pair $(\rho_1,\rho_2)$ so that $\rho_2 = \rho$. But there is no way to handle any more coordinates.

One aspect of this explanation is that it implies that the reason we can handle two coordinates instead of just one is that we are considering the Hausdorff dimension, which corresponds to a lim inf, rather than the packing dimension, which corresponds to a lim sup. It is well-known that the Hausdorff and packing dimensions of a self-affine set can be different even in two dimensions; see e.g. \cite[Theorem 4.6]{LalleyGatzouras} together with \cite[Proposition 2.2(i)]{Peres4}. This is in contrast to the situation for finite conformal IFSes, where the Hausdorff and packing dimensions are always the same \cite[Lemma 3.14]{MauldinUrbanski1}.
\end{remark}

\section{Construction of dimension gap sponges}
\label{sectioncounterexample}

In this section we prove the main result of this paper, the existence of sponges with a dimension gap, viz. Theorem \ref{theoremcounterexample}. Before starting the proof, we give a sketch to convey the main ideas. In the sketch we write down formulas without giving any justification, since these formulas will be justified in detail in the real proof.

\begin{convention}
We denote the product of two matrices $\bfA$ and $\bfB$ by $\bfA\cdot\bfB$. It should not be confused with the scalar product of two vectors $\vv$ and $\ww$, which we denote by $\lb \vv,\ww\rb$.
\end{convention}

\begin{figure}
\scalebox{0.75}{
\begin{tikzpicture}[line cap=round,line join=rounr]
\clip(-1,-1) rectangle (9,9);
\fillsquare{0.16}{0.64}{0.5}{1.75};
\fillsquare{0.16}{0.64}{2.25}{3.5};
\fillsquare{0.8}{1.28}{0.5}{1.75};
\fillsquare{0.8}{1.28}{2.25}{3.5};
\fillsquare{1.44}{1.92}{0.5}{1.75};
\fillsquare{1.44}{1.92}{2.25}{3.5};
\fillsquare{2.08}{2.56}{0.5}{1.75};
\fillsquare{2.08}{2.56}{2.25}{3.5};
\fillsquare{2.72}{3.2}{0.5}{1.75};
\fillsquare{2.72}{3.2}{2.25}{3.5};
\fillsquare{3.36}{3.84}{0.5}{1.75};
\fillsquare{3.36}{3.84}{2.25}{3.5};
\fillsquare{4.2}{4.95}{4.25}{5.25};
\fillsquare{5.15}{5.9}{4.25}{5.25};
\fillsquare{6.10}{6.85}{4.25}{5.25};
\fillsquare{7.05}{7.8}{4.25}{5.25};
\fillsquare{4.2}{4.95}{5.5}{6.5};
\fillsquare{5.15}{5.9}{5.5}{6.5};
\fillsquare{6.10}{6.85}{5.5}{6.5};
\fillsquare{7.05}{7.8}{5.5}{6.5};
\fillsquare{4.2}{4.95}{6.75}{7.75};
\fillsquare{5.15}{5.9}{6.75}{7.75};
\fillsquare{6.10}{6.85}{6.75}{7.75};
\fillsquare{7.05}{7.8}{6.75}{7.75};
\vertline080;
\vertline088;
\horizline080;
\horizline088;
\end{tikzpicture}
}
\caption{An example of the disjoint-union-of-product-IFSes construction, with $\#(D) = \#(J) = 2$. In the actual proof of Theorem \ref{theoremcounterexample} we have $\#(D) = \#(J) = 3$.}
\label{figurebasicconstruction}
\end{figure}
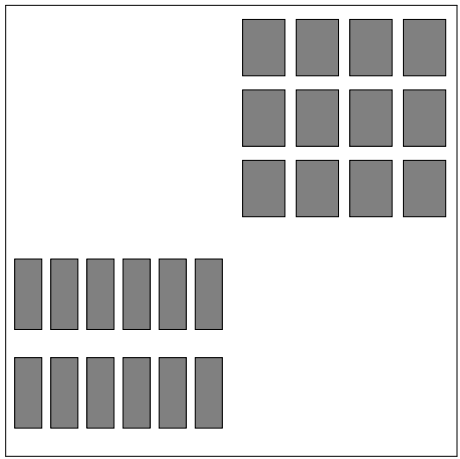

\begin{proof}[Proof Sketch of Theorem \ref{theoremcounterexample}]
The goal is to find a diagonal IFS $\Phi = (\phi_a)_{a\in E}$ on $[0,1]^3$ and a cycle $\rr\in\RR$ such that $\HD(\nu_\rr) > \DynD(\Phi)$. The IFS will be of a special form: it will be the disjoint union of three sub-IFSes, each of which will be the direct product of three similarity IFSes on $[0,1]$ (cf. Figure \ref{figurebasicconstruction}). Letting $D = J = \{1,2,3\}$, we can write $\Phi = \coprod_{j\in J} \prod_{i\in D} \Phi_{i,j}$, where for each $i\in D$ and $j\in J$, $\Phi_{i,j} = (\phi_{i,j,a})_{a\in E_{i,j}}$ is a similarity IFS on $[0,1]$ consisting of similarities all with the same contraction ratio. The properties of the overall IFS $\Phi$ are determined up to some fudge factors by the entropy and Lyapunov exponents of the component IFSes $\Phi_{i,j}$ ($i\in D$, $j\in J$), which we denote by $H_{i,j}$ and $X_{i,j}$, respectively. (In the actual proof, the entropy and Lyapunov exponent of $\Phi_{i,j}$ will only be approximately proportional to $H_{i,j}$ and $X_{i,j}$, rather than equal.) The matrices $\bfH = (H_{i,j})$ and $\bfX = (X_{i,j})$ can be more or less arbitrary, subject to the restriction that $0 < H_{i,j} < X_{i,j}$, which describes the fact that the dimension of the limit set of $\Phi_{i,j}$ must be strictly between 0 and 1. To make the overall IFS satisfy the coordinate ordering condition, the further restriction $X_{i,j} < X_{i + 1,j}$ is also needed.

Once the relation between $\Phi$ and the matrices $\bfH$ and $\bfX$ has been established, $\DynD(\Phi)$ can be estimated based on $\bfH$ and $\bfX$. The maximum of the function $\pp\mapsto\delta(\pp)$ is always attained at points of the form $\sum_{j\in J} q_j \uu_j$, where $\uu_j$ denotes the normalized uniform measure on $E_j \df \prod_{i\in D} E_{i,j}$, i.e. $\uu_j = \#(E_j)^{-1} \sum_{\aa\in E_j} \delta_\aa$, and $\qq = (q_1,q_2,q_3) \in \R^J$ is a probability vector. Equivalently, the maximum is attained at $\bfM\cdot\qq$ for some $\qq\in\Delta$, where $\Delta \subset \R^J$ is the space of probability vectors on $J$ and $\bfM\cdot\ee_j = \uu_j$ for all $j\in J$. Here and hereafter $(\ee_j)_{j\in J}$ denotes the standard basis of $\R^J$. To make things simpler later, we will choose $\bfH$ and $\bfX$ so that we can be even more precise: the maximum of $\pp\mapsto \delta(\pp)$ is attained at $\pp = \bfM\cdot\uu$, where $\uu = (1/3,1/3,1/3)\in\Delta$ is the normalized uniform measure on $J$.

Next, let us describe the cycle $\rr\in\RR$ for which we will prove that $\HD(\nu_\rr) > \DynD(\Phi)$. Its range will consist of probability vectors of the form $\bfM\cdot\qq$ with $\qq\in\Delta$, i.e. those probability vectors which were considered candidates for the maximum of $\pp\mapsto\delta(\pp)$ in the previous paragraph. So we can write $\rr_b = \bfM\cdot\ss_b$, where $\ss:(0,\infty)\to\Delta$ is exponentially periodic. The trajectory of $\ss$ will be the inscribed circle of the triangle $\Delta$ (cf. Figure \ref{figureinscribed}), and the exponential period of $\ss$ will be $e^{2\pi\gamma}$ for some small number $\gamma > 0$. Formally, we will write
\[
\ss_{\exp(\gamma t)} = \zz(t)
\]
where $\zz:\R\to \Delta$ is a unit speed (with respect to angle) parameterization of the inscribed circle of $\Delta$. (In the actual proof, for greater generality we will let $\rho$ denote the period of $\zz$, so that in our case $\rho = 2\pi$.)

\begin{figure}
\scalebox{1}{
\begin{tikzpicture}[line cap=round,line join=round,>=triangle 45,x=1.0cm,y=1.0cm]
\clip(-3.798698493920108,-2.5896412504082704) rectangle (3.973381390913784,4.084586335592283);
\draw (0.0,3.5)-- (0.0,-0.0);
\draw (0.0,-0.0)-- (-3.0,-2.0);
\draw (0.0,-0.0)-- (3.0,-2.0);
\draw (-1.4966531143275326,-0.9977687428850217)-- (0.0,2.030608364992254);
\draw (0.0,2.030608364992254)-- (1.4962628069604793,-0.9975085379736529);
\draw (1.4962628069604793,-0.9975085379736529)-- (-1.4966531143275326,-0.9977687428850217);
\draw [line width=0.58pt,dash pattern=on 2pt off 2pt,fill=black,fill opacity=0.1] (1.721392000161447E-4,-0.06776266351843657) circle (0.9298759758774434cm);
\begin{scriptsize}
\draw[color=black] (0.12055546509869226,3.7358391612728115) node {$q_1$};
\draw [fill=black] (-0.006,-0.094) circle (0.67pt);
\draw[color=black] (-0.006,-0.27) node {$\uu$};
\draw[color=black] (-3.1,-1.759) node {$q_2$};
\draw[color=black] (3.1,-1.759) node {$q_3$};
\draw [fill=black] (-1.4966531143275326,-0.9977687428850217) circle (0.5pt);
\draw [fill=black] (0.0,2.020608364992254) circle (0.5pt);
\draw [fill=black] (1.4962628069604793,-0.9975085379736529) circle (0.5pt);
\end{scriptsize}
\end{tikzpicture}
}
\caption{The inscribed circle of the simplex $\Delta$, which represents the trajectory of $\ss$. This trajectory geometrically represents the non-invariant/pseudo-Bernoulli measure that we prove has dimension strictly greater than the dynamical dimension, while its center $\uu$ represents the invariant/Bernoulli measure of maximal dimension.}
\label{figureinscribed}
\end{figure}
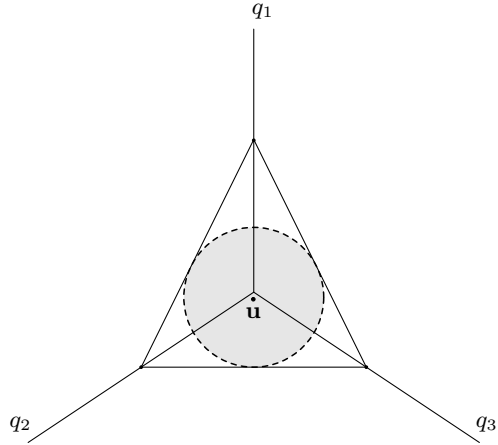

\begin{remark}
The fact that the trajectory of $\ss$ is a circle is motivated by the fact that $\ss$ should be (exponentially) periodic and smooth, and that the ``center'' of its trajectory should be the maximum of $\pp\mapsto\delta(\pp)$. The fact that the exponential period is close to 1 is motivated by the fact that the ``advantage'' that non-constant cycles $\rr\in\RR$ have over constant points $\pp\in\PP$ is the fact that they are ``moving'', so to maximize this advantage, it makes sense to maximize the speed of motion. However, the tradeoff is that the dimension gap $\HD(\nu_\rr) - \DynD(\Phi)$ ends up depending proportionally on $\gamma$ as $\gamma\to 0$ (see \eqref{gapasymp} below), so the size of the dimension gap tends to zero as $\gamma \to 0$. This is one of the reasons that it is difficult for us to get good lower bounds on the size of the dimension gap; cf. Questions \ref{questionMDG}.
\end{remark}

With this setup, after making the additional simplification that $\bfH_i\cdot\uu = 2^{i - 1}$ and $\bfX_i\cdot\uu = 2^i$ for all $i\in D$, where $\bfH_i$ and $\bfX_i$ denote the $i$th rows of $\bfH$ and $\bfX$, respectively, one finds that the size of the dimension gap is
\begin{equation}
\label{gapasymp}
\HD(\nu_\rr) - \DynD(\Phi) = \gamma \inf_{t\in [0,2\pi]}\sum_{i\in D} \bfK_i\cdot\bfZ(t_{i,0}) + O(\gamma^2),
\end{equation}
where $\bfK_i = 2^{-i} (\bfH_i - (1/2)\bfX_i)$, $\bfZ:\R\to\R^J$ is a unit speed parameterization of a certain circle in the plane $P = \{\qq\in\R^J : q_1 + q_2 + q_3 = 0\}$, and $t_{i,0}$ is defined by the equation
\begin{equation}
\label{tidef1}
t = t_{i,0} + \bfY_i\cdot\bfZ(t_{i,0}) \all i,
\end{equation}
where $\bfY_i = 2^{-i} \bfX_i$. So the goal now is to make the coefficient of $\gamma$ in \eqref{gapasymp} positive, while still making sure that the maximum of $\pp\mapsto \delta(\pp)$ is attained at $\pp = \bfM\cdot\uu$. This is done most efficiently by assuming that $\bfK$ and $\bfY$ are close to a known value that would lead to the map $\Delta \ni \qq \mapsto\delta(\bfM\cdot\qq)$ being constant; i.e.
\begin{align*}
\bfK &= \epsilon\w\bfK,&
\bfY &= \bfU + \epsilon\w\bfY,
\end{align*}
where $\bfU$ is the $3\times 3$ matrix whose entries are all equal to 1, $\w\bfK$ and $\w\bfY$ are matrices chosen so that $\w\bfK\cdot\uu = \w\bfY\cdot\uu = 0$, and $\epsilon > 0$ is small. Then the time $t_{i,0}$ defined by \eqref{tidef1} approaches $t$ as $\epsilon\to 0$, so the coefficient of $\gamma$ in \eqref{gapasymp} becomes
\[
\epsilon^2\inf_{t\in [0,2\pi]} \w\bfK_i\cdot\bfZ'(t)[-\w\bfY_i\cdot\bfZ(t)] + O(\epsilon^3),
\]
so we need the coefficient of $\epsilon^2$ in this expression to be positive:
\begin{equation}
\label{noncommA}
\sup_{t\in [0,2\pi]} (\w\bfK_i\cdot\bfZ'(t))(\w\bfY_i\cdot\bfZ(t)) < 0.
\end{equation}
At the same time, we need the maximum of $\pp\mapsto \delta(\pp)$ to be attained at $\pp = \bfM\cdot\uu$; it is enough to check that
\begin{align}
\label{commA}
\sum_{i\in D} \w\bfK_i &= \0,&
\sum_{i\in D} (\w\bfK_i\cdot\qq)(\w\bfY_i\cdot\qq) > 0 \all \qq\in\R^J\butnot\R\uu.
\end{align}
The proof is then completed by finding matrices $\w\bfK$ and $\w\bfY$ that satisfy all these requirements. Intuitively, the difficulty should come in reconciling the requirements \eqref{noncommA} and \eqref{commA}, since the latter is what shows that a constant element of $\Delta$ cannot produce a dimension greater than $3/2$, while the former is what shows that the nonconstant circular cycle can produce such a dimension gap. However, the requirements are compatible because \eqref{noncommA} incorporates the geometry of circular motion, in which the derivative $\bfZ'(t)$ is always orthogonal to $\bfZ(t)$, while \eqref{commA} cannot incorporate the geometry of any shape because it comes from considering only constant cycles. This completes the proof sketch.
\end{proof}

\begin{proof}[Proof of Theorem \ref{theoremcounterexample}]
It suffices to consider the case $d = 3$, since a 3-dimensional Bara\'nski sponge can be isometrically embedded into any higher dimension. Let $\bfH = (H_{i,j})$ and $\bfX = (X_{i,j})$ be $3\times 3$ matrices to be specified later. We think of their rows as being indexed by the set $D \df \{1,2,3\}$, while their columns are indexed by $J \df \{1,2,3\}$.
Here we have made a conceptual distinction between the sets $D$ and $J$ even though they are set-theoretically the same, because the fact that these two sets have the same cardinality has no relevance until much later in the argument. Geometrically, $D$ corresponds to the number of dimensions (i.e. $D$ is the set of coordinates), while $J$ corresponds to the number of distinct ``types'' of contractions that we will put into our diagonal IFS. We will assume that
\begin{equation}
\label{basicassumptions}
\begin{split}
&0 < H_{i,j} < X_{i,j} \all i\in D \all j\in J,\\
&X_{i,j} < X_{i + 1,j} \all i = 1,2 \all j\in J.
\end{split}
\end{equation}
Fix $k$ large. For each $i\in D$ and $j\in J$, let $N_{i,j} = \lfloor e^{k H_{i,j}}\rfloor$ and $r_{i,j} = e^{-k X_{i,j}}$, and let $\Phi_{i,j} = (\phi_{i,j,a})_{a\in E_{i,j}}$ be a one-dimensional IFS of contracting similarities satisfying the strong separation condition with respect to $[0,1]$ such that
\begin{itemize}
\item[(I)] $\phi_{i,j,a}([0,1]) \subset ((j - 1)/3,j/3)$  for all $a\in E_{i,j}$;
\item[(II)] $\#(E_{i,j}) = N_{i,j}$; and
\item[(III)] $|\phi_{i,j,a}'| = r_{i,j}$ for all $a\in E_{i,j}$.
\end{itemize}
This is possible as long as $N_{i,j} r_{i,j} < 1/3$, which is true for all sufficiently large $k$, since by hypothesis $H_{i,j} < X_{i,j}$.

Now for each $j \in J$, let $E_j = \prod_{i \in D} E_{i,j}$ and $\Phi_j = (\phi_{j,\aa})_{\aa\in E_j}$, where $\phi_{j,\aa}(\xx) = (\phi_{i,j,a_i}(x_i))_{i\in D}$. Let
\[
E = \coprod_{j \in J} E_j = \{(j,\aa) : j\in J, \aa\in E_j\},
\]
and consider the IFS $\Phi = (\phi_{j,\aa})_{(j,\aa)\in E}$. Note that the second half of condition \eqref{basicassumptions} guarantees that $\Phi$ satisfies the coordinate ordering condition with respect to the identity permutation. To emphasize the dependence of $\Phi$ on the parameter $k$, we will sometimes write $\Phi_k$ instead of $\Phi$.

We proceed to estimate $\DynD(\Phi)$ and $\HD(\Phi)$.

{\it Estimation of $\DynD(\Phi)$.}
For each $i\in D$ and $j \in J$, let $\Perm(E_{i,j})$ denote the group of permutations of $E_{i,j}$. Then the group $G = \prod_{j \in J} \prod_{i \in D} \Perm(E_{i,j})$ admits a natural action on $E$, with respect to which the functions $h_I$ ($I\subset D$) and $\chi_i$ ($i\in D$) are invariant. Now let $\pp$ be any probability measure on $E$, and let $\w\pp = \mu_G\ast\pp$, where $\mu_G$ is the Haar/uniform measure of $G$ and $\ast$ denotes convolution. Note that $\w\pp$ is $G$-invariant. Since $h_I$ is superlinear and $\chi_i$ is linear, we have $h_I(\w\pp) \geq h_I(\pp)$ and $\chi_i(\w\pp) = \chi_i(\pp)$ for all $I$ and $i$. Consequently, it follows  from \eqref{deltap1} that $\delta(\w\pp) \geq \delta(\pp)$, so the supremum in \eqref{DDbernoulli} can be taken over the class of $G$-invariant measures on $E$. Such measures are of the form
\[
\pp = \sum_{j \in J} q_j \uu_j,
\]
where $\qq = (q_1,q_2,q_3)$ is a probability vector on $J$, and $\uu_j$ denotes the normalized uniform measure on $E_j$, i.e. $\uu_j = \#(E_j)^{-1} \sum_{\aa\in E_j} \delta_\aa$. Equivalently, $\pp = \bfM\cdot\qq$, where $\bfM:\R^J \to \R^E$ is the linear operator such that $\bfM \cdot \ee_j = \uu_j$ for all $j\in J$. Note that for all $I \subset D$, by Lemma \ref{lemmaentropy} we have
\begin{equation}
\label{hIp}
\begin{split}
h_I(\bfM\cdot\qq) &= \sum_{j \in J} q_j h_I(\uu_j) + O(1)
= \sum_{j \in J} q_j \sum_{i\in I} \log(N_{i,j}) + O(1)\\
&= \sum_{j \in J} q_j \sum_{i\in I} k H_{i,j} + O(1)
= k \sum_{i\in I} \bfH_i \cdot \qq + O(1),\footnotemark
\end{split}
\end{equation}
and for all $i\in D$\Footnotetext{Here $O(1)$ denotes a quantity whose magnitude is bounded by a constant, in particular independent of $k$.}
\begin{align}
\label{chiip}
\chi_i(\bfM\cdot\qq) &= \sum_{j \in J} q_j \chi_i(\uu_j)
= \sum_{j \in J} q_j k X_{i,j}
= k \bfX_i\cdot\qq.
\end{align}
Here $\bfH_i$ and $\bfX_i$ denote the $i$th rows of $\bfH$ and $\bfX$, respectively, i.e. $\bfH_i = \ee_i^*\cdot\bfH$ and $\bfX_i = \ee_i^*\cdot\bfX$, where $(\ee_i^*)_{i\in D}$ is the dual of the standard basis of $\R^d$. So by \eqref{deltap2}, we have
\begin{equation}
\label{deltaSk}
\DynD(\Phi_k)
= \max_{\qq\in\Delta} \sum_{i\in D} \frac{k \bfH_i\cdot \qq + O(1)}{k \bfX_i\cdot\qq}
 \tendsto{k\to\infty} \delta_0 \df \max_{\qq\in\Delta} \sum_{i\in D} \frac{ \bfH_i\cdot \qq}{\bfX_i\cdot\qq},
\end{equation}
where $\Delta$ denotes the space of probability vectors on $J$.

{\it Estimation of $\HD(\Phi)$.} Fix a continuous map $\zz:\R\to\Delta$ of period $\rho > 0$, to be determined later. Fix $\gamma > 0$ small, and let $\ss:(0,\infty)\to \Delta$ be defined by the formula
\[
\ss_b = \zz(\log(b)/\gamma).
\]
Next, let $\rr_b = \bfM \cdot \ss_b$ for all $b > 0$. Note that $\rr$ is exponentially $\lambda$-periodic, where $\lambda = e^{\gamma\rho}$. We will estimate $\HD(\Phi)$ from below by estimating $\HD(\nu_{\rr})$.

Fix $t\in [0,\rho]$, and for each $i\in D$ let $B_i > 0$ be given by the formula
\begin{equation}
\label{Bidef2}
e^{\gamma t} = \bfX_i \cdot \bfS_{B_i},
\end{equation}
where $\bfS_B \df \int_0^B \ss_b\;\dee b$. Applying \eqref{chiip} with $\qq = \bfS_{B_i}$ shows that \eqref{Bidef} is satisfied with $B = k e^{\gamma t}$. It follows that
\[
\delta(\rr,k e^{\gamma t}) = \frac{1}{k e^{\gamma t}} \int h(\{i\in D : b \leq B_i\};\rr_b)\;\dee b.
\]
Now by \eqref{hIp},
\begin{align*}
h(\{i\in D :b\leq B_i\};\rr_b)
&= \sum_{i:b\leq B_i} k\bfH_i\cdot\ss_b + O(1),
\end{align*}
and since the left hand side is zero whenever $b > \max_i B_i$, we can add parentheses in the last expression:
\[
h(\{i\in D : b\leq B_i\};\rr_b) = \sum_{i:b\leq B_i} [k\bfH_i\cdot\ss_b + O(1)].
\]
So we have
\begin{align*}
\delta(\rr,k e^{\gamma t})
&= \frac{1}{k e^{\gamma t}} \int \sum_{i:b\leq B_i} [k\bfH_i\cdot\ss_b + O(1)] \;\dee b\\
&= \frac{1}{ke^{\gamma t}} \sum_{i\in D} \int_0^{B_i} [k\bfH_i\cdot\ss_b + O(1)] \;\dee b\\
&= \sum_{i\in D} \frac{k\bfH_i\cdot\bfS_{B_i} + O(B_i)}{k\bfX_i\cdot\bfS_{B_i}}\cdot \by{\eqref{Bidef2}}
\end{align*}
Since $B_i/\bfX_i\cdot\bfS_{B_i}$ is bounded independent of $k$, we have
\[
\delta(\rr,k e^{\gamma t}) \tendsto{k\to\infty} \delta(\gamma;t) \df \sum_{i\in D} \frac{\bfH_i\cdot \bfS_{B_i}}{\bfX_i\cdot \bfS_{B_i}},
\]
and the convergence is uniform with respect to $t$. So by Theorem \ref{theoremHDcompute}
\begin{equation}
\label{HDSk}
\HD(\Phi_k) \geq \delta(\rr) = \inf_{t\in [0,\rho]} \delta(\rr,k e^{\gamma t}) \tendsto{k\to\infty} \delta_\gamma \df \inf_{t\in [0,\rho]} \delta(\gamma;t).
\end{equation}
By \eqref{deltaSk} and \eqref{HDSk}, to complete the proof we must show that $\delta_\gamma > \delta_0$ if $\gamma$ is small enough. It suffices to show that
\begin{equation}
\label{ETSgamma}
\lim_{\gamma\to 0} \frac{\delta_\gamma - \delta_0}{\gamma} > 0.
\end{equation}
{\it Taking the limit $\gamma\to 0$.}
In the sequel, we will make the following assumptions about the matrices $\bfH$ and $\bfX$:
\begin{equation}
\label{maxloc}
\text{the maximum in \eqref{deltaSk} occurs at $\qq = \uu \df (1/3,1/3,1/3)$,}
\end{equation}
\begin{align}
\label{maxvalues}
\bfH_i\cdot\uu &= 2^{i - 1} \all i,&
\bfX_i\cdot\uu &= 2^i \all i.
\end{align}
We remark that it follows from these assumptions that $\delta_0 = 3/2$. We also assume that
\begin{equation}
\label{avgy}
\frac{1}{\rho}\int_0^\rho \zz(t) \;\dee t = \uu
\end{equation}
and that $\gamma = \log(2)/(\ell\rho)$ for some $\ell\in\N$. Before proceeding further, let us estimate $\bfS_{\exp(\gamma t)}^{(\gamma)}$. Here, we have notated the dependence of $\bfS$ on $\gamma$, since it is relevant to what follows. Let $\bfZ:\R\to\R^J$ be the unique antiderivative of $\zz - \uu$ such that $\int_0^\rho \bfZ(t)\;\dee t = 0$. Note that by \eqref{avgy}, $\bfZ$ is periodic of period $\rho$.

\begin{claim}
We have
\begin{equation}
\label{Sgammaz}
\bfS_{\exp(\gamma t)}^{(\gamma)} = e^{\gamma t} [\uu + \gamma \bfZ(t) + O(\gamma^2)]
\end{equation}
as $\gamma\to 0$.
\end{claim}
\begin{subproof}
For convenience, we write $\w\zz(t) = \zz(t) - \uu$, $\w\ss_b^{(\gamma)} = \ss_b^{(\gamma)} - \uu$, and $\w\bfS_B^{(\gamma)} = \bfS_B^{(\gamma)} - B\uu$. Since $\ss$ is exponentially $e^{\gamma\rho}$-periodic, we have $\w\bfS_{\exp(\gamma(t + \rho))}^{(\gamma)} = e^{\gamma\rho}\w\bfS_{\exp(\gamma t)}^{(\gamma)}$, so
\begin{align*}
\w\bfS_{\exp(\gamma t)}^{(\gamma)}
&= \frac{\w\bfS_{\exp(\gamma(t + \rho))}^{(\gamma)} - \w\bfS_{\exp(\gamma t)}^{(\gamma)}}{e^{\gamma\rho} - 1}
= \frac{1}{e^{\gamma\rho} - 1} \int_{e^{\gamma t}}^{e^{\gamma(t + \rho)}} \w\ss_b\;\dee b\noreason\\
&= \frac{1}{\gamma\rho + O(\gamma^2)} \int_0^\rho \gamma e^{\gamma(t + s)} \w\zz(t + s)\;\dee s\\
&= \frac{1}{\rho + O(\gamma)} \int_0^\rho [e^{\gamma (t + s)} - e^{\gamma t}] \w\zz(t + s) \;\dee s \by{\eqref{avgy}}\\
&= e^{\gamma t} \left[\frac{\gamma}{\rho} \int_0^\rho s \w\zz(t + s)\;\dee s + O(\gamma^2)\right].
\end{align*}
Thus \eqref{Sgammaz} holds for the function
\begin{equation}
\label{Zdef2}
\bfZ(t) = \frac{1}{\rho} \int_0^\rho s \w\zz(t + s)\;\dee s.
\end{equation}
Integration by parts shows that $\bfZ'(t) = \w\zz(t)$, and Fubini's theorem shows that $\int_0^\rho \bfZ(t)\;\dee t = 0$, with both calculations using \eqref{avgy}. So the function $\bfZ$ defined by \eqref{Zdef2} is the same as the function $\bfZ$ defined earlier.
\end{subproof}

Let $t_i = \log(2^i B_i)/\gamma$. Then
\begin{align*}
e^{\gamma(t - t_i)}
&= e^{-\gamma t_i}\bfX_i\cdot\bfS_{B_i}^{(\gamma)} \by{\eqref{Bidef2}}\\
&= 2^{-i} e^{-\gamma t_i}\bfX_i\cdot\bfS_{\exp(\gamma t_i)}^{(\gamma)} \since{$\log(2) \in \N\gamma\rho$}\\
&= 2^{-i} \bfX_i \cdot\big(\uu + \gamma \bfZ(t_i) + O(\gamma^2)\big)\by{\eqref{Sgammaz}}\\
&= 1 + 2^{-i} \gamma \bfX_i\cdot\bfZ(t_i) + O(\gamma^2). \by{\eqref{maxvalues}}
\end{align*}
In particular $e^{\gamma(t - t_i)} = 1 + O(\gamma)$, which implies that $t - t_i = O(1)$ and thus we can use the Taylor expansion on the left-hand side:
\[
1 + \gamma (t - t_i) + O(\gamma^2) = 1 + \gamma \bfY_i\cdot\bfZ(t_i) + O(\gamma^2),
\]
where $\bfY_i = 2^{-i} \bfX_i$. Let us write $t_i = t_{i,\gamma}$ to remind ourselves that $t_i$ depends on $\gamma$. We have
\[
t = t_{i,\gamma} + \bfY_i\cdot\bfZ(t_{i,\gamma}) + O(\gamma).
\]
So if we let $t_{i,0}\in\R$ be the solution to the equation
\[
t = t_{i,0} + \bfY_i\cdot\bfZ(t_{i,0}),
\]
then $t_{i,\gamma} = t_{i,0} + O(\gamma)$. This is because the derivative of the right-hand side with respect to $t_{i,0}$ is bounded from below:
\[
1 + \bfY_i\cdot\bfZ'(t_{i,0})
= \bfY_i\cdot\uu + \bfY_i\cdot\w\zz(t_{i,0})
= \bfY_i\cdot\zz(t_{i,0}) \geq \min_{\pp\in\Delta} \bfY_i\cdot\pp = \min_{j\in J} Y_{i,j} > 0.
\]
Next, let $\bfK_i = 2^{-i}(\bfH_i - (1/2)\bfX_i)$. Then
\begin{align*}
\frac{\delta(\gamma;t) - \delta_0}{\gamma} &= \frac{1}{\gamma}\left[\sum_{i\in D} \frac{\bfH_i\cdot\bfS_{B_i}^{(\gamma)}}{\bfX_i\cdot \bfS_{B_i}^{(\gamma)}} - \frac{3}{2}\right]\\
&= \frac{1}{\gamma}\sum_{i\in D} \frac{\bfK_i\cdot\bfS_{B_i}^{(\gamma)}}{\bfY_i\cdot \bfS_{B_i}^{(\gamma)}}\\
&= \frac{1}{\gamma}\sum_{i\in D} \frac{\bfK_i\cdot\bfS_{2^i B_i}^{(\gamma)}}{\bfY_i\cdot \bfS_{2^i B_i}^{(\gamma)}} \since{$\log(2) \in \N\gamma\rho$}\\
&= \frac{1}{\gamma}\sum_{i\in D} \frac{\bfK_i\cdot\big(\uu + \gamma\bfZ(t_{i,\gamma}) + O(\gamma^2)\big)}{\bfY_i\cdot \big(\uu + \gamma\bfZ(t_{i,\gamma}) + O(\gamma^2)\big)} \by{\eqref{Sgammaz}}\\
&= \frac{1}{\gamma}\sum_{i\in D} \frac{\gamma\bfK_i\cdot\bfZ(t_{i,\gamma}) + O(\gamma^2)}{1 + O(\gamma)} \by{\eqref{maxvalues}}\\
&= \sum_{i\in D} \bfK_i\cdot\bfZ(t_{i,0}) + O(\gamma)
\end{align*}
and thus
\[
\lim_{\gamma\to 0}\frac{\delta(\gamma;t) - \delta_0}{\gamma} = \gammacoeff(t) \df \sum_{i\in D} \bfK_i\cdot\bfZ(t_{i,0}),
\]
and the convergence is uniform with respect to $t$. So to complete the proof, we must show that there exist matrices $\bfH$ and $\bfX$ satisfying \eqref{basicassumptions}, \eqref{maxloc}, and \eqref{maxvalues}, such that for some periodic function $\zz:\R\to\Delta$ satisfying \eqref{avgy}, we have
\begin{equation}
\label{infq}
\inf_{t\in [0,\rho]} \gammacoeff(t) > 0.
\end{equation}
{\it Constructing the matrices $\bfH$ and $\bfX$; letting $\epsilon \to 0$.}
To construct these matrices, let $\bfU$ be the $3\times 3$ matrix whose entries are all equal to 1, and fix $\epsilon > 0$ small to be determined. We will let
\begin{align*}
\bfH_i &= 2^i \bfK_i + (1/2) \bfX_i,&
\bfX_i &= 2^i \bfY_i,&
\bfK &= \epsilon\w\bfK,&
\bfY &= \bfU + \epsilon\w\bfY,
\end{align*}
where $\w\bfK$ and $\w\bfY$ will be chosen later, with the property that
\begin{equation}
\label{avgzero}
\w\bfK\cdot\uu = \w\bfY\cdot\uu = \0.
\end{equation}
Then \eqref{maxvalues} is easily verified, and if $\epsilon$ is small enough then \eqref{basicassumptions} holds. Now for $\qq\in\Delta$, we have
\begin{equation}
\label{dimdifference}
\begin{split}
\sum_{i\in D} \frac{\bfH_i\cdot\qq}{\bfX_i\cdot\qq} - \frac{3}{2}
&= \sum_{i\in D} \frac{\bfK_i\cdot\qq}{\bfY_i\cdot\qq}
= \sum_{i\in D} \frac{\epsilon\w\bfK_i\cdot\qq}{1 + \epsilon\w\bfY_i\cdot\qq}\\
&= \sum_{i\in D} \left[\epsilon\w\bfK_i\cdot\qq - \epsilon^2 (\w\bfK_i\cdot\qq)(\w\bfY_i\cdot\qq)\right] + O(\epsilon^3 \cdot \|\qq - \uu\|^3).
\end{split}
\end{equation}
To demonstrate that \eqref{maxloc} holds, we need to show that \eqref{dimdifference} is non-positive for all $\qq\in\Delta$. To show that this is true whenever $\epsilon$ is sufficiently small, it suffices to show that
\begin{equation}
\label{avgKzero}
\sum_{i\in D} \w\bfK_i = \0
\end{equation}
and
\begin{equation}
\label{comm}
\sum_{i\in D} (\w\bfK_i\cdot\qq)(\w\bfY_i\cdot\qq) > 0 \all \qq\in \R^J\butnot\R\uu.
\end{equation}
Finally, to show that \eqref{infq} holds whenever $\epsilon$ is sufficiently small, we introduce subscripts to indicate the dependence on $\epsilon$ of all quantities that depend on $\epsilon$. We have
\[
t = t_{i,\epsilon} + \bfY_{i,\epsilon}\cdot \bfZ(t_{i,\epsilon}) = t_{i,\epsilon} + \bfU_i\cdot\bfZ(t_{i,\epsilon}) + \epsilon\w\bfY_i\cdot \bfZ(t_{i,\epsilon}).
\]
The middle term is zero, since $\bfU\cdot\w\zz(t) = \bfU\cdot\zz(t) - \bfU\cdot\uu = (1,1,1) - (1,1,1) = \0$ for all $t\in\R$. Thus
\[
t = t_{i,\epsilon} + \epsilon\w\bfY_i\cdot \bfZ(t_{i,\epsilon}).
\]
So in particular, $t_{i,\epsilon} = t + O(\epsilon)$, and thus
\[
t = t_{i,\epsilon} + \epsilon\w\bfY_i\cdot\bfZ(t) + O(\epsilon^2).
\]
Thus
\begin{align*}
\gammacoeff_\epsilon(t) &= \sum_{i\in D} \bfK_i\cdot\bfZ\big(t - \epsilon\w\bfY_i\cdot\bfZ(t) + O(\epsilon^2)\big)\\
&= \sum_{i\in D} \epsilon\w\bfK_i\cdot\left[\bfZ\big(t - \epsilon\w\bfY_i\cdot\bfZ(t) + O(\epsilon^2)\big) - \bfZ(t)\right] \by{\eqref{avgKzero}}\\
&= \sum_{i\in D} \big(\epsilon\w\bfK_i\cdot\bfZ'(t)\big)\big(-\epsilon\w\bfY_i\cdot\bfZ(t)\big) + O(\epsilon^3).
\end{align*}
(Note that in this step, we use the fact that $\zz$ is continuous (and thus $\bfZ$ is $C^1$); it is not enough for $\zz$ to be piecewise continuous.) So it is enough to show that
\begin{equation}
\label{noncomm}
\sum_{i\in D} \big(\w\bfK_i\cdot\bfZ'(t)\big)\big(\w\bfY_i\cdot\bfZ(t)\big) < 0 \all t\in [0,\rho].
\end{equation}
{\it Constructing $\w \bfK$, $\w \bfY$, and $\zz$.}
Until now, we have not used the fact that $d = 3$, nor the fact that $\#(D)$ and $\#(J)$ are equal, except as a convenience of notation. But now, we construct explicit matrices $\w\bfK$ and $\w\bfY$ and an explicit continuous periodic function $\zz:\R\to\Delta$ that satisfy \eqref{avgy}, \eqref{avgzero}, \eqref{avgKzero}, \eqref{comm}, and \eqref{noncomm}:
\begin{align*}
\w\bfK &= \left[\begin{array}{ccc}
1 && -1\\
-1 & 1 &\\
& -1 & 1
\end{array}\right],\hspace{.5 in}
\w\bfY = \left[\begin{array}{ccc}
1 & -1 &\\
& 1 & -1\\
-1 && 1
\end{array}\right],\\
\zz(t) &= \frac{1}{3}\left(1 + \cos(t),1 + \cos\left(t + \frac{2\pi}{3}\right),1 + \cos\left(t + \frac{4\pi}{3}\right)\right),\\
\bfZ(t) &= \frac{1}{3}\left(\sin(t),\sin\left(t + \frac{2\pi}{3}\right),\sin\left(t + \frac{4\pi}{3}\right)\right).
\end{align*}
Now \eqref{avgy}, \eqref{avgzero}, and \eqref{avgKzero} are immediate. Although it is possible to verify \eqref{comm} and \eqref{noncomm} by direct computation, we give a geometrical proof. First note that $\w\bfK$ and $\w\bfY$ both commute with the group $G$ of orientation-preserving permutation matrices. It follows that the quadratic form $Q_1(\qq) = \sum_{i\in D} (\w\bfK_i\cdot\qq)(\w\bfY_i\cdot\qq)$ is invariant under $G$, and thus the conic section $\{\qq \in P: Q_1(\qq) = \pm 1\}$ is also invariant under $G$, where $P$ is the plane through the origin parallel to $\Delta$, i.e. $P = \{(q_1,q_2,q_3) \in \R^3 : q_1 + q_2 + q_3 = 0\}$. Now if this conic section is a non-circular ellipse, then its major axis must be fixed by $G$, and if it is a hyperbola, then the asymptotes must be either fixed or interchanged. All of these scenarios are impossible because $G$ is of order 3 and has no fixed lines in $P$, so the conic section is a circle and thus $Q_1(\qq) = c_1 \|\qq\|^2$ for some constant $c_1$. The sign of $c_1$ can be calculated by taking the trace of $Q_1$, i.e. $3c_1 = \sum_{i\in D} \lb \w\bfK_i,\w\bfY_i\rb = 3$. Geometrically, this formula is a consequence of the fact that the angle between $\w\bfK_i$ and $\w\bfY_i$ is 60 degrees, and their magnitudes are both $\sqrt 2$. This demonstrates \eqref{comm}.

Next, observe that the path traced by $\bfZ$ is a circle in $P$ centered at the origin, with the opposite orientation from the triangular path $\ee_1\to\ee_2\to\ee_3\to\ee_1$.\Footnote{Although we have checked that the signs and orientations in this paragraph are correct (and we thank the referee for pointing out a couple of errors in a previous version), it is not necessary to check this to verify the validity of the argument; cf. Remark \ref{remarksignsirrelevant}.} Thus, for all $t\in\R$ we have $\bfZ'(t) = \vv\times \bfZ(t)$, where $\times$ denotes the cross product and $\vv = -\sqrt 3\uu = -\frac{\sqrt 3}{3}(1,1,1)$ is a unit vector. So if $\bfN$ denotes the $3\times 3$ matrix such that $\bfN\cdot\xx = \vv\times\xx$ for all $\xx\in\R^3$, then the left-hand side of \eqref{noncomm} is equal to
\[
\sum_{i\in D} \big(\w\bfK_i \cdot\bfN\cdot \bfZ(t)\big) \big(\w\bfY_i\cdot \bfZ(t)\big)
\]
and so what is needed is to show that the quadratic form
\[
Q_2(\qq) = \sum_{i\in D} \big(\w\bfK_i \cdot\bfN\cdot \qq\big) \big(\w\bfY_i\cdot \qq\big)
\]
is negative definite on $P$. Now since $\bfN$ is a rotation of the plane $P$, it commutes with $G$, so the argument of the preceding paragraph can be used to show that $Q_2(\qq) = c_2 \|\qq\|^2$ for some constant $c_2$ whose sign is the same as the sign of the trace of $Q_2$, i.e. $3c_2 = \sum_{i\in D} \lb \w\bfK_i\cdot\bfN,\w\bfY_i\rb = -3\sqrt 3$. Geometrically, this formula is a consequence of the fact that the angle between $\w\bfK_i\cdot\bfN$ and $\w\bfY_i$ is 150 degrees, and their magnitudes are both $\sqrt 2$. This demonstrates \eqref{noncomm}.
\end{proof}

\begin{remark}
It is not hard to see why it is impossible to construct matrices $\w\bfK$ and $\w\bfY$ as well as a periodic function $\zz$ satisfying the relevant formulas unless $\#(D),\#(J) \geq 3$. Indeed. if $\#(J) \leq 2$, then $\Delta$ is a one-dimensional space, and so by the intermediate value theorem we have $\zz(t) = 0$ for some $t$, rendering \eqref{noncomm} impossible. Similarly, if $\#(D) \leq 2$, then by \eqref{avgKzero} we have $\w\bfK_2 = -\w\bfK_1$, and again by the intermediate value theorem we have $\w\bfK_1\cdot\zz(t) = 0$ for some $t$. Thus again, \eqref{noncomm} is impossible in this case.
\end{remark}

\begin{remark}
\label{remarksignsirrelevant}
It should be pointed out that the directions of the inequalities \eqref{comm} and \eqref{noncomm} are irrelevant to the question of whether there exist $\w\bfK$, $\w\bfY$, and $\zz$ satisfying them. Indeed, if $\w\bfK$ (or $\w\bfY$) is replaced by its negative, then the signs of both inequalities simultaneously flip, while if $\zz$ is replaced by the function $t\mapsto \zz(-t)$, then the sign of \eqref{noncomm} flips but the sign of \eqref{comm} stays the same. So given a triple $(\w\bfK,\w\bfY,\zz)$ that satisfies \eqref{comm} and \eqref{noncomm} with respect to any given direction of signs, it is possible to modify this triple in a minor way to get a triple that satisfies \eqref{comm} and \eqref{noncomm} with respect to the correct direction of signs.
\end{remark}

\section{Open questions}
\label{sectionopen}

Although Theorem \ref{theoremcounterexample} provides an answer to Question \ref{mainquestion} in dimensions 3 and higher, it is natural to ask what happens in dimension 2:

\begin{questions}
If $X \subset \R^2$ is a compact set and $T:X\to X$ is an expanding map satisfying the specification property, then is the Hausdorff dimension of $X$ equal to the supremum of the Hausdorff dimensions of the ergodic $T$-invariant measures? And if so, is the supremum attained, and what are the properties of the measure attaining the supremum? What if the specification property is not assumed?
\end{questions}

Although we have proven that the dimension gap $\HD(\Phi) - \DynD(\Phi)$ is strictly positive, we cannot get a very good lower bound on its size. This leads to some natural questions:

\begin{questions}
\label{questionMDG}
Given $d \geq 3$, what is
\[
\text{MDG}(d) \df \sup_{\Lambda_\Phi \subset [0,1]^d} \big(\HD(\Phi) - \DynD(\Phi)\big),
\]
where the supremum is taken over all Bara\'nski sponges $\Lambda_\Phi$? (Here MDG is short for ``maximal dimension gap''.) Is the answer any different if the supremum is restricted to sponges that satisfy the coordinate ordering condition? And what about the related quantity
\[
\text{MDG}'(d) \df \sup_{\Lambda_\Phi \subset [0,1]^d} \frac{\HD(\Phi) - \DynD(\Phi)}{\DynD(\Phi)}?
\]
In our proofs it seems that this quantity is more natural to consider than $\text{MDG}(d)$; for example, we can show that $\text{MDG}'(d) \leq d - 2$ for all $d \geq 2$ (Theorem \ref{theoremDGbound} above). To avoid the effects of low dimension, we ask: what is the asymptotic behavior of $\text{MDG}'(d)$ as $d \to \infty$? For example, is it bounded or unbounded?
\end{questions}

Although Theorem \ref{theoremcontinuous} shows that the map $\Phi\mapsto \HD(\Phi)$ is continuous on the space of Bara\'nski sponges, in many contexts the Hausdorff dimension is not only continuous but real-analytic (see e.g. \cite{AndersonRocha, Pollicott, RoyUrbanski, Ruelle, Rugh}). So we ask:

\begin{questions}
Is the function $\Phi\mapsto \HD(\Phi)$ real-analytic, or at least piecewise real-analytic, on the space of Bara\'nski sponges? What about the subclass of strongly Bara\'nski sponges?
\end{questions}

Finally, we speculate that the key ideas behind our definition of a pseudo-Bernoulli measure might apply more generally. We therefore ask the following questions:

\begin{questions}
Is there any useful class of measures that exhibits scale-dependent behavior similar to pseudo-Bernoulli measures in a more general context? For example, can the ideas of this paper be used to construct repellers with a dimension gap other than sponges?
\end{questions}

\bibliographystyle{amsplain}

\bibliography{bibliography}

\providecommand{\bysame}{\leavevmode\hbox to3em{\hrulefill}\thinspace}
\providecommand{\MR}{\relax\ifhmode\unskip\space\fi MR }
\providecommand{\MRhref}[2]{%
  \href{http://www.ams.org/mathscinet-getitem?mr=#1}{#2}
}
\providecommand{\href}[2]{#2}
\begin{thebibliography}{10}

\bibitem{AndersonRocha}
James~W. Anderson and Andr{\'e}~C. Rocha, \emph{Analyticity of {H}ausdorff
  dimension of limit sets of {K}leinian groups}, Ann. Acad. Sci. Fenn. Math.
  \textbf{22} (1997), no.~2, 349--364. \MR{1469796}

\bibitem{AvilaLyubich2}
Artur Avila and Mikhail Lyubich, \emph{Lebesgue measure of {F}eigenbaum {J}ulia
  sets}, \url{http://arxiv.org/abs/1504.02986}, preprint 2015.

\bibitem{Baranski}
Krzysztof Bara{\'n}ski, \emph{Hausdorff dimension of the limit sets of some
  planar geometric constructions}, Adv. Math. \textbf{210} (2007), no.~1,
  215--245. \MR{2298824 (2008e:28016)}

\bibitem{Baranski2}
\bysame, \emph{Hausdorff dimension of self-affine limit sets with an invariant
  direction}, Discrete Contin. Dyn. Syst. \textbf{21} (2008), no.~4,
  1015--1023. \MR{2399447}

\bibitem{Barreira_book1}
Luis Barreira, \emph{Dimension and recurrence in hyperbolic dynamics}, Progress
  in Mathematics, vol. 272, Birkh\"auser Verlag, Basel, 2008. \MR{2434246}

\bibitem{Barreira_book2}
Lu{\'{\i}}s Barreira, \emph{Dimension theory of hyperbolic flows}, Springer
  Monographs in Mathematics, Springer, Cham, 2013. \MR{3087567}

\bibitem{BarreiraGelfert}
Luis Barreira and Katrin Gelfert, \emph{Dimension estimates in smooth dynamics:
  a survey of recent results}, Ergodic Theory Dynam. Systems \textbf{31}
  (2011), no.~3, 641--671. \MR{2794942}

\bibitem{BarreiraWolf}
Luis Barreira and Christian Wolf, \emph{Measures of maximal dimension for
  hyperbolic diffeomorphisms}, Comm. Math. Phys. \textbf{239} (2003), no.~1-2,
  93--113. \MR{1997117}

\bibitem{Barreira}
Luis~M. Barreira, \emph{A non-additive thermodynamic formalism and applications
  to dimension theory of hyperbolic dynamical systems}, Ergodic Theory Dynam.
  Systems \textbf{16} (1996), no.~5, 871--927. \MR{1417767}

\bibitem{Bedford}
Tim Bedford, \emph{Crinkly curves, {M}arkov partitions and box dimensions in
  self-similar sets}, Ph.D. thesis, The University of Warwick, 1984.

\bibitem{BenoistQuint_book}
Yves Benoist and Jean-Fran{\c{c}}ois Quint, \emph{Random walks on reductive
  groups}, Ergebnisse der Mathematik und ihrer Grenzgebiete. 3. Folge. A Series
  of Modern Surveys in Mathematics [Results in Mathematics and Related Areas.
  3rd Series. A Series of Modern Surveys in Mathematics], vol.~62, Springer,
  Cham, 2016. \MR{3560700}

\bibitem{Bowen_book}
Rufus Bowen, \emph{Equilibrium states and the ergodic theory of {A}nosov
  diffeomorphisms}, Lecture Notes in Mathematics, vol. 470, Springer-Verlag,
  Berlin, 2008.

\bibitem{ChenPesin}
Jianyu Chen and Yakov Pesin, \emph{Dimension of non-conformal repellers: a
  survey}, Nonlinearity \textbf{23} (2010), no.~4, R93--R114. \MR{2602012}

\bibitem{DenkerUrbanski2}
Manfred Denker and Mariusz Urba{\'n}ski, \emph{On {S}ullivan's conformal
  measures for rational maps of the {R}iemann sphere}, Nonlinearity \textbf{4}
  (1991), no.~2, 365--384. \MR{1107011 (92f:58097)}

\bibitem{EinsiedlerWard}
Manfred Einsiedler and Thomas Ward, \emph{Ergodic theory with a view towards
  number theory}, Graduate Texts in Mathematics, vol. 259, Springer-Verlag
  London, Ltd., London, 2011. \MR{2723325}

\bibitem{Falconer4}
Kenneth Falconer, \emph{Dimensions of self-affine sets: a survey}, Further
  developments in fractals and related fields, Trends Math.,
  Birkh\"auser/Springer, New York, 2013, pp.~115--134. \MR{3184190}

\bibitem{Feng}
De-Jun Feng, \emph{Equilibrium states for factor maps between subshifts}, Adv.
  Math. \textbf{226} (2011), no.~3, 2470--2502. \MR{2739782}

\bibitem{FengHu}
De-Jun Feng and Huyi Hu, \emph{Dimension theory of iterated function systems},
  Comm. Pure Appl. Math. \textbf{62} (2009), no.~11, 1435--1500. \MR{2560042}

\bibitem{Fraser2}
Jonathan~M. Fraser, \emph{On the packing dimension of box-like self-affine sets
  in the plane}, Nonlinearity \textbf{25} (2012), no.~7, 2075--2092.
  \MR{2947936}

\bibitem{FraserHowroyd}
Jonathan~M. Fraser and Douglas Howroyd, \emph{Assouad type dimensions for
  self-affine sponges}, \url{http://arxiv.org/abs/1508.03393}, preprint 2015.

\bibitem{FriedlandOchs}
Shmuel Friedland and Gunter Ochs, \emph{Hausdorff dimension, strong
  hyperbolicity and complex dynamics}, Discrete Contin. Dynam. Systems
  \textbf{4} (1998), no.~3, 405--430. \MR{1612732}

\bibitem{GatzourasPeres_survey}
Dimitrios Gatzouras and Yuval Peres, \emph{The variational principle for
  {H}ausdorff dimension: a survey}, Ergodic theory of {${\bf Z}^d$} actions
  ({W}arwick, 1993--1994), London Math. Soc. Lecture Note Ser., vol. 228,
  Cambridge Univ. Press, Cambridge, 1996, pp.~113--125. \MR{1411217}

\bibitem{GatzourasPeres2}
\bysame, \emph{Invariant measures of full dimension for some expanding maps},
  Ergodic Theory Dynam. Systems \textbf{17} (1997), no.~1, 147--167.
  \MR{1440772}

\bibitem{Hutchinson}
John Hutchinson, \emph{Fractals and self-similarity}, Indiana Univ. Math. J.
  \textbf{30} (1981), no. 5, 713--747.

\bibitem{Kaenmaki2}
Antti K{\"a}enm{\"a}ki, \emph{On natural invariant measures on generalised
  iterated function systems}, Ann. Acad. Sci. Fenn. Math. \textbf{29} (2004),
  no.~2, 419--458. \MR{2097242}

\bibitem{KenyonPeres}
Richard Kenyon and Yuval Peres, \emph{Measures of full dimension on
  affine-invariant sets}, Ergodic Theory Dynam. Systems \textbf{16} (1996),
  no.~2, 307--323. \MR{1389626 (98m:28042)}

\bibitem{KotusUrbanski4}
Janina Kotus and Mariusz Urba{\'n}ski, \emph{Geometry and ergodic theory of
  non-recurrent elliptic functions}, J. Anal. Math. \textbf{93} (2004),
  35--102. \MR{2110325 (2005j:37065)}

\bibitem{LalleyGatzouras}
Steven~P. Lalley and Dimitrios Gatzouras, \emph{Hausdorff and box dimensions of
  certain self-affine fractals}, Indiana Univ. Math. J. \textbf{41} (1992),
  no.~2, 533--568. \MR{1183358 (93j:28011)}

\bibitem{LedrappierYoung2}
Fran{\c c}ois Ledrappier and Lai-Sang Young, \emph{The metric entropy of
  diffeomorphisms. {II}. {R}elations between entropy, exponents and dimension},
  Ann. of Math. (2) \textbf{122} (1985), no. 3, 540--574.

\bibitem{Luzia}
Nuno Luzia, \emph{Measure of full dimension for some nonconformal repellers},
  Discrete Contin. Dyn. Syst. \textbf{26} (2010), no.~1, 291--302. \MR{2552788}

\bibitem{MauldinUrbanski1}
R.~Daniel Mauldin and Mariusz Urba{\'n}ski, \emph{Dimensions and measures in
  infinite iterated function systems}, Proc. London Math. Soc. (3) \textbf{73}
  (1996), no. 1, 105--154.

\bibitem{MauldinUrbanski2}
\bysame, \emph{Graph directed {M}arkov systems: Geometry and dynamics of limit
  sets}, Cambridge Tracts in Mathematics, vol. 148, Cambridge University Press,
  Cambridge, 2003.

\bibitem{McCluskeyManning}
Heather McCluskey and Anthony Manning, \emph{Hausdorff dimension for
  horseshoes}, Ergodic Theory Dynam. Systems \textbf{3} (1983), no.~2,
  251--260. \MR{742227}

\bibitem{McMullen_carpets}
Curt McMullen, \emph{The {H}ausdorff dimension of general {S}ierpinski
  carpets}, Nagoya Math. J. \textbf{96} (1984), 1--9.

\bibitem{Neunhauserer}
J{\"o}rg Neunh{\"a}userer, \emph{Number theoretical peculiarities in the
  dimension theory of dynamical systems}, Israel J. Math. \textbf{128} (2002),
  267--283. \MR{1910385}

\bibitem{Olsen1}
Lars Olsen, \emph{Self-affine multifractal {S}ierpinski sponges in {$\mathbb
  R^d$}}, Pacific J. Math. \textbf{183} (1998), no.~1, 143--199. \MR{1616626}

\bibitem{Olsen2}
\bysame, \emph{Symbolic and geometric local dimensions of self-affine
  multifractal {S}ierpinski sponges in {$\mathbb R^d$}}, Stoch. Dyn. \textbf{7}
  (2007), no.~1, 37--51. \MR{2303792}

\bibitem{Peres4}
Yuval Peres, \emph{The packing measure of self-affine carpets}, Math. Proc.
  Cambridge Philos. Soc. \textbf{115} (1994), no.~3, 437--450. \MR{1269931}

\bibitem{PeresSolomyak}
Yuval Peres and Boris Solomyak, \emph{Problems on self-similar sets and
  self-affine sets: an update}, Fractal geometry and stochastics, {II}
  ({G}reifswald/{K}oserow, 1998), Progr. Probab., vol.~46, Birkh\"auser, Basel,
  2000, pp.~95--106. \MR{1785622}

\bibitem{Petersen2}
Karl Petersen, \emph{Information compression and retention in dynamical
  processes}, Dynamics and randomness ({S}antiago, 2000), Nonlinear Phenom.
  Complex Systems, vol.~7, Kluwer Acad. Publ., Dordrecht, 2002, pp.~147--217.
  \MR{1975578}

\bibitem{Pollicott}
Mark Pollicott, \emph{Analyticity of dimensions for hyperbolic surface
  diffeomorphisms}, Proc. Amer. Math. Soc. \textbf{143} (2015), no.~8,
  3465--3474. \MR{3348789}

\bibitem{PrzytyckiRivera}
Feliks Przytycki and Juan Rivera-Letelier, \emph{Statistical properties of
  topological {C}ollet-{E}ckmann maps}, Ann. Sci. \'Ec. Norm. Sup\'er. (4)
  \textbf{40} (2007), no.~1, 135--178.

\bibitem{PrzytyckiUrbanski}
Feliks Przytycki and Mariusz Urba{\'n}ski, \emph{Conformal fractals: ergodic
  theory methods}, London Mathematical Society Lecture Note Series, 371,
  Cambridge University Press, Cambridge, 2010.

\bibitem{QianXie}
Min Qian and Jian-Sheng Xie, \emph{Entropy formula for endomorphisms: relations
  between entropy, exponents and dimension}, Discrete Contin. Dynam. Systems
  \textbf{21} (2008), no.~2, 367--392. \MR{2385697}

\bibitem{Reeve}
Henry W.~J. Reeve, \emph{Infinite non-conformal iterated function systems},
  Israel J. Math. \textbf{194} (2013), no.~1, 285--329. \MR{3047072}

\bibitem{RogersTaylor}
Claude~A. Rogers and Stephen~J. Taylor, \emph{The analysis of additive set
  functions in {E}uclidean space}, Acta Math. \textbf{101} (1959), 273--302.

\bibitem{RoyUrbanski}
Mario Roy and Mariusz Urba{\'n}ski, \emph{Real analyticity of {H}ausdorff
  dimension for higher dimensional hyperbolic graph directed {M}arkov systems},
  Math. Z. \textbf{260} (2008), no.~1, 153--175. \MR{2413348}

\bibitem{Ruelle}
David Ruelle, \emph{Repellers for real analytic maps}, Ergodic Theory Dynamical
  Systems \textbf{2} (1982), no.~1, 99--107. \MR{684247}

\bibitem{Rugh}
Hans~H. Rugh, \emph{On the dimensions of conformal repellers. {R}andomness and
  parameter dependency.}, Ann. of Math. (2) \textbf{168} (2008), no. 3,
  695--748.

\bibitem{Schmeling}
J{\"o}rg Schmeling, \emph{Ergodic theory: fractal geometry}, Mathematics of
  complexity and dynamical systems. {V}ols. 1--3, Springer, New York, 2012,
  pp.~288--301. \MR{3220676}

\bibitem{SchmelingWeiss}
J{\"o}rg Schmeling and Howard Weiss, \emph{An overview of the dimension theory
  of dynamical systems}, Smooth ergodic theory and its applications ({S}eattle,
  {WA}, 1999), Proc. Sympos. Pure Math., vol.~69, Amer. Math. Soc., Providence,
  RI, 2001, pp.~429--488. \MR{1858542}

\bibitem{Urbanski6}
Mariusz Urba{\'n}ski, \emph{Rational functions with no recurrent critical
  points}, Ergodic Theory Dynam. Systems \textbf{14} (1994), no.~2, 391--414.

\bibitem{UrbanskiZdunik3}
Mariusz Urba{\'n}ski and Anna Zdunik, \emph{Geometry and ergodic theory of
  non-hyperbolic exponential maps}, Trans. Amer. Math. Soc. \textbf{359}
  (2007), no.~8, 3973--3997. \MR{2302520}

\bibitem{Weihrauch}
Klaus Weihrauch, \emph{Computable analysis}, Texts in Theoretical Computer
  Science. An EATCS Series, Springer-Verlag, Berlin, 2000, An introduction.
  \MR{1795407}

\bibitem{Yayama}
Yuki Yayama, \emph{Dimensions of compact invariant sets of some expanding
  maps}, Ergodic Theory Dynam. Systems \textbf{29} (2009), no.~1, 281--315.
  \MR{2470637}

\bibitem{Young3}
Lai-Sang Young, \emph{What are {SRB} measures, and which dynamical systems have
  them?}, J. Statist. Phys. \textbf{108} (2002), no.~5-6, 733--754, Dedicated
  to David Ruelle and Yasha Sinai on the occasion of their 65th birthdays.
  \MR{1933431}

\end{thebibliography}

\end{document}